\documentclass[smallextended]{amsart}
\usepackage[top=1in, bottom=1in, left=1.25in, right=1.25in]{geometry}
\usepackage{mathrsfs}
\usepackage{stackrel,amssymb}
\usepackage{amsfonts}
\usepackage[linkcolor=red,citecolor=blue]{hyperref}
\usepackage{latexsym}
\usepackage[all,cmtip]{xy}
\usepackage{amsmath}
\usepackage{color,ifsym}
\usepackage{graphicx}
\usepackage{fancybox}
\usepackage{longtable}
\usepackage{lscape}
\usepackage{float}
\usepackage{mathdots}
\usepackage{enumerate}
\usepackage{multirow}

\newtheorem{thm}{Theorem}
\newtheorem{lem}[thm]{Lemma}        
\newtheorem{cor}[thm]{Corollary}
\newtheorem{prop}[thm]{Proposition}

\newtheorem{mthm}[thm]{Main Theorem}

\newtheorem{conj}[thm]{Conjecture}

\newtheorem{rem}{{\it Remark}} 

\theoremstyle{plain} 
\newcommand{\thistheoremname}{}
\newtheorem*{genericthm*}{\thistheoremname}
\newenvironment{namedthm*}[1]
{\renewcommand{\thistheoremname}{#1}%
	\begin{genericthm*}}
	{\end{genericthm*}}

\begin{document}
	\title[holomorphy of intertwining operators]{holomorphy of normalized intertwining operators for certain induced representations I: a toy example}
   
	\author{CAIHUA LUO}
	\address{department of mathematics \\ bar-ilan university \\ ramat-gan, 5290000\\ israel}
	\email{chluo@amss.ac.cn}
	\date{}
\begin{abstract}	
The theory of intertwining operators plays an important role in the development of the Langlands program. This, in some sense, is a very sophisticated theory, but the basic question of its singularity, in general, is quite unknown. Motivated by its deep connection with the longstanding pursuit of constructing automorphic $L$-functions via the method of integral representations, we prove the holomorphy of normalized local intertwining operators, normalized in the sense of Casselman--Shahidi, for a family of induced representations of quasi-split classical groups as an exercise. Our argument is the outcome of an observation of an intrinsic non-symmetry property of normalization factors appearing in different reduced decompositions of intertwining operators. Such an approach bears the potential to work in general.
\end{abstract}
\maketitle

\section{introduction}
Let $G_n$ be a rank $n$ quasi-split classical group defined over a $p$-adic field $F$, $\sigma_{\bar{r}}$ be a generic discrete series of a small rank group $G_{n_0}$, and $\rho_{c}(\tau_a)$ be a Speh representation associated to the supercuspidal representation $\tau$ of some $GL$, a basic question in the theory of intertwining operators is to determine the singularity of standard intertwining operators, especially the following one
\[M_{c,a}(s,~\tau,~\sigma_{\bar{r}}):~|det(\cdot)|^s\rho_{c}(\tau_a)\rtimes \sigma_{\bar{r}}\longrightarrow |det(\cdot)|^{-s}\rho_{c}(\tau_a)\rtimes \sigma_{\bar{r}}  \]
(Please refer to the main body of the paper for the notions). Inspired by the profound Langlands--Shahidi theory, a conjectural normalization factor $\alpha(s,~\tau_a,~\sigma_{\bar{r}})$ for the case $c=1$ has been proposed by Casselman--Shahidi in \cite{casselman1998irreducibility}. In the paper, following their spirit, we define a normalization factor $\alpha_{c,a}(s,~\tau,~\sigma_{\bar{r}})$ of the intertwining operator $M_{c,a}(s,~\tau,~\sigma_{\bar{r}})$ for the general case, and prove the holomorphy of the normalized version (see Main Theorem \ref{mainthm}), i.e.,
\[M^*_{c,a}(s,~\tau,~\sigma_{\bar{r}}):=\frac{1}{\alpha_{c,a}(s,~\tau,~\sigma_{\bar{r}})}M_{c,a}(s,~\tau,~\sigma_{\bar{r}}). \]   
The proof is mainly by induction in nature with the help of some simple reduction arguments. The novel part of our contribution is an observation of an intrinsic non-symmetry property of those normalized intertwining operators $\alpha_{c,a}(\cdots)$ appearing in the reduced decompositions of the intertwining operator $M_{c,a}(\cdots)$. This enables us to reduce the holomorphicity problem to some low rank cases that we can handle by hand or certain unitary cases which follows from M{\oe}glin's work.

Historically, more attention have been focused on the full normalized intertwining operators in the sense of Langlands (to name a few, see \cite{moeglin1989residue,moeglin2010holomorphy,moeglin2011intertwiningeisenstein,zhang1997localintertwining,cogdellkimshapiroshahidi2004,jantzenkim2001intertwiningimage,jiangliuzhang2013}), roughly speaking the following one
\[N_{c,a}(s,~\tau,~\sigma_{\bar{r}}):=\frac{\beta_{c,a}(s,~\tau,~\sigma_{\bar{r}})}{\alpha_{c,a}(s,~\tau,~\sigma_{\bar{r}})}M_{c,a}(s,~\tau,~\sigma_{\bar{r}}). \]
However, a special case of $M^*_{c,a}(s,~\tau,~\sigma_{\bar{r}})$ when $a=1$ and $\sigma_{\bar{r}}=1$, i.e., the degenerate principal series case, has been vastly investigated by Harris--Kudla--Sweet among others in the development of theta correspondence (see \cite{piatetskishapirotriple,kudlarallis1992,kudlasweet1997,harriskudlasweet1996,sweet1995,ikeda1992location,yamana2011,yamanatheta}). On the other hand, certain cases have recently gained the attention in the pursuit of constructing automorphic $L$-functions via the theory of integral representations, for example Rankin--Selberg method, Doubling method and its generalized version, Braverman--Kazhdan/Ng{\^o} program (see \cite{yamanatheta,kaplan2013gcd,cai2019doubling,cai2018doubling,getzliu2021refined,jiangluozhang2020Symplectic}). We hope our argument illustrated in the paper could play a role in these aspects as does in the aforementioned situations. Let us end the introduction by pointing out that our holomoprhy result proved here and extended in a future work could reprove many holomorphicity results for the full normalized intertwining operators obtained earlier by M{\oe}glin and many others.   

The structure of the paper is as follows. In the next section, we prepare some notions and state our Main Theorem \ref{mainthm}, while its proof will be given in the following sections gradually, i.e., Degenerate case $GL$ (Section \ref{caseGL}), Reduction steps (Section \ref{redsteps}), Base case $\rho_{c}(\tau)\rtimes \sigma_{r}$ (Section \ref{case1}), and Base case $\tau_a\rtimes\sigma_{r}$ (Section \ref{case2}). 

\section{main theorem}
Let $F$ be a non-archimedean local field of characteristic zero. Denote by $\mathcal{O}$ its
ring of integers and by $\mathcal{P}$ its maximal ideal generated by the uniformizer $\mathfrak{w}$. Let $q$ be the number of elements in the residue field $\mathcal{O}/\mathcal{P}$. We set $|\cdot|$ to be the absolute value of $F$ such that $|\mathfrak{w}|=q^{-1}$. Let $E/F$ be a quadratic field extension of $F$, we use the same terminologies of $F$ for $E$ if there is no confusion.

\subsection{Groups}Let $G_n$ be a quasi-split classical group of rank $n$ defined over $F$. They are the $F$-quasi-split unitary groups $U_{2n}$ and $U_{2n+1}$ of hermitian type associated to $E/F$, the $F$-split (special) odd orthogonal group $(S)O_{2n+1}$ and the symplectic group $Sp_{2n}$, and the $F$-quasi-split even (special) orthogonal group $(S)O_{2n}$. One may notice that the metaplectic group $Mp_{2n}$, the unique two-fold cover of $Sp_{2n}$, can also be dealt with in this setting, but we do not want to include it here as it may increase the burden of introducing additional notions.
\subsection{Certain induced representations}For a partition $n=akc+n_0$ for some positive integers $a,~k,~c$ and a non-negative integer $n_0$, we denote by $P_{akc}=M_{akc}N_{akc}$ the standard parabolic subgroup of $G_n$ with its Levi subgroup $M_{akc}\simeq GL_{akc}\times G_{n_0}$. Let $\tau$ be a fixed unitary supercuspidal representation of $GL_k$, we set $\tau_a$ to be the associated Steinberg representation of $GL_{ak}$, i.e., the unique subrepresentation of the normalized induced representation
\[\tau|det(\cdot)|^{\frac{a-1}{2}}\times\tau|det(\cdot)|^{\frac{a-3}{2}}\times \cdots\times \tau|det(\cdot)|^{-\frac{a-1}{2}}  :=Ind^{GL_{ak}}(\tau|det(\cdot)|^{\frac{a-1}{2}}\otimes\cdots\otimes \tau|det(\cdot)|^{-\frac{a-1}{2}}).\]
In particular, $\tau_1=\tau$. For the discrete series representation $\tau_a$ of $GL_{ak}$, we denote $\rho_c(\tau_a)$ to be the associated Speh representation of $GL_{akc}$, i.e., the unique Langlands quotient of the standard module
\[\tau_a|det(\cdot)|^{\frac{c-1}{2}}\times\tau_a|det(\cdot)|^{\frac{c-3}{2}}\times \cdots\times \tau_a|det(\cdot)|^{-\frac{c-1}{2}},\]
which can also be viewed as the unique subrepresentation of
\[\tau_a|det(\cdot)|^{-\frac{c-1}{2}}\times\tau_a|det(\cdot)|^{-\frac{c-3}{2}}\times \cdots\times \tau_a|det(\cdot)|^{\frac{c-1}{2}}.\]
Let $n_{00}$ be a non-negative integer and $\sigma$ be a fixed generic supercuspidal representation of $G_{n_{00}}$ (please refer to \cite{shahidi1990proof} or \cite{luo2021casselmanshahidiconjecture} for the definition of the standard notion ``generic''), we say an irreducible admissible representation $\sigma_{\bar{r}}$ of $G_{n_0}$ is supported on $\tau$ and $\sigma$ if the unitary part of the Bernstein--Zelevinsky data of $\sigma_{\bar{r}}$ is contained in $\{\tau,~\tau^*,~\sigma\}$, i.e., $\sigma_{\bar{r}}$ is a constituent of the normalized parabolic induction representation
\[|det(\cdot)|^{s_1}\tau\times|det(\cdot)|^{s_2}\tau\times\cdots\times |det(\cdot)|^{s_t}\tau\rtimes\sigma:=Ind^{G_{n_0}}(|det(\cdot)|^{s_1}\tau\otimes|det(\cdot)|^{s_2}\tau\otimes\cdots\otimes |det(\cdot)|^{s_t}\tau\otimes\sigma) \]
for some real numbers $s_i\in \mathbb{R}$, $i=1,~\cdots,~t$. Here $\tau^*$ is the (conjugate) contragredient representation of $\tau$. 

For an irreducible admissible representation $\sigma_{\bar{r}}$ of $G_{n_0}$ and a Speh representation $\rho_{c}(\tau_a)$ of $GL_{akc}$, we define the associated normalized parabolic induction of $G_n$ as follows, for $s\in\mathbb{C}$,
\[\rho_c(\tau_a)|det(\cdot)|^s\rtimes \sigma_{\bar{r}}:=Ind^{G_n}_{P_{akc}}(\rho_c(\tau_a)|det(\cdot)|^s\otimes \sigma_{\bar{r}}). \]
In the paper, we will only consider the case that $\sigma_{\bar{r}}$ is a generic discrete series representation of $G_{n_0}$. Indeed, the condition that $\sigma_{\bar{r}}$ is a discrete series representation supported on $\tau$ and $\sigma$ implies readily that $\tau^*\simeq \tau$, i.e., $\tau$ is (conjugate) self-dual (see \cite{tadic1998regular,heiermann2004decomposition,luo2018R}).
 
\subsection{Intertwining operator} For $\sigma_{\bar{r}}$ a generic discrete series representation of $G_{n_0}$,
we let $M_{c,a}(s,~\tau,~\sigma_{\bar{r}})$ be the following standard intertwining operator
\[M_{c,a}(s,~\tau,~\sigma_{\bar{r}}):~\rho_c(\tau_a)|det(\cdot)|^s\rtimes\sigma_{\bar{r}}\longrightarrow\rho_c(\tau_a)^*|det(\cdot)|^{-s}\rtimes\sigma_{\bar{r}} \]
defined by the continuation of the integral \cite{waldspurger2003formule}, for $f_s(g)\in \rho_c(\tau_a)|det(\cdot)|^s\rtimes\sigma_{\bar{r}}$,
\[\int_{N_{kc}}f_s(w_cug)du. \]
Here $\rho_c(\tau_a)^*$ is the (conjugate) contragredient of $\rho_c(\tau_a)$ which is isomorphic to $\rho_c(\tau^*_a)$, and the Weyl element $w_c\in G_n$ is given by, $N_0=2n_0$ or $2n_0+1$ depends on $G_{n_0}$,
\[w_c:=(-1)^{akc}\begin{pmatrix}
&  & I_{akc} \\ 
& I_{N_0} &  \\ 
\pm I_{akc} &  & 
\end{pmatrix}. \]
One may note that the above $w_c$ does not always exist for $SO_{2n}$. One can slightly modify $w_c$ to fix this issue as done in \cite{luo2021casselmanshahidiconjecture}, but we can also consider its equivalent form $$J_{\bar{P}_{akc}|P_{akc}}(s,~\tau,~\sigma_{\bar{r}}):~Ind^{G_n}_{P_{akc}}(\rho_c(\tau_a)|det(\cdot)|^s\otimes \sigma_{\bar{r}})\longrightarrow Ind^{G_n}_{\bar{P}_{akc}}(\rho_c(\tau_a)|det(\cdot)|^s\otimes \sigma_{\bar{r}}), $$
which is given by the continuation of the integral, for $f_s(g)\in \rho_c(\tau_a)|det(\cdot)|^s\rtimes\sigma_{\bar{r}}$,
\[\int_{\bar{N}_{akc}}f_s(ug)du. \]
Here $\bar{P}_{akc}=M_{akc}\bar{N}_{akc}$ is the unique opposite parabolic subgroup of $P_{akc}$. Indeed, we have
\[M_{c,a}(s,~\tau,~\sigma_{\bar{r}})=J_{P_{akc}|\bar{P}_{akc}}(-s,~\tau^*,~\sigma_{\bar{r}}) \circ Ad_{w_c}\]
if such a $w_c$ exists, i.e., $w_c^{-1}P_{akc}w_c=\bar{P}_{akc}$, where
\begin{align*}
Ad_{w_c}: ~Ind^{G_n}_{P_{akc}}(\rho_c(\tau_a)|det(\cdot)|^s\otimes \sigma_{\bar{r}})&\stackrel{\sim}{\longrightarrow} Ind^{G_n}_{\bar{P}_{akc}}(\rho_c(\tau^*_a)|det(\cdot)|^{-s}\otimes \sigma_{\bar{r}})\\
f_s(g)  & \mapsto (g\mapsto f_s(w_cg)).
\end{align*}
In view of this, we will always use the form $M_{c,a}(s,~\tau,~\sigma_{\bar{r}})$ for simplicity throughout the paper.
\subsection{Main Theorem}
Inspired by the profound Langlands--Shahidi theory, following Casselman--Shahidi's normalization of intertwining operators \cite{shahidi1990proof}, we define $M^*_{c,a}(s,~\tau,~\sigma_{\bar{r}})$ to be
\[\frac{1}{\alpha_{c,a}(s,~\tau,~\sigma_{\bar{r}})}M_{c,a}(s,~\tau,~\sigma_{\bar{r}}), \]
and 
\[ \alpha_{c,a}(s,~\tau,~\sigma_{\bar{r}}):=\prod_{i=1}^{\lceil\frac{c}{2}\rceil}L(2s-1-c+2i,\tau_a,\rho)\prod_{i=1}^{\lfloor\frac{c}{2}\rfloor}L(2s-c+2i,\tau_a,\rho^-)L(s-\frac{c-1}{2},\tau_a\times \sigma_{\bar{r}}), \]
where for $n_0=0$, $L(s,\tau_a\times\sigma_{\bar{r}}):=L(s,\tau_a)$ if $G_n=Sp_{2n}$ and $U_{2n+1}$, 1 otherwise. Here $\rho$ (resp. $\rho^-$) is defined as follows.
\[\rho:=\begin{cases}
\mbox{Asai},& \mbox{if } G_n=U_{2n},\\
\mbox{Asai}\otimes \chi_{E/F}, &\mbox{if } G_n=U_{2n+1},\\
Sym^2,& \mbox{if } G_n=(S)O_{2n+1},\\
\Lambda^2,& \mbox{if } G_n=Sp_{2n}\mbox{ or }(S)O_{2n};
\end{cases}\]
\[\rho^-:=\begin{cases}
\mbox{Asai}\otimes \chi_{E/F},& \mbox{if } G_n=U_{2n},\\
\mbox{Asai}, &\mbox{if } G_n=U_{2n+1},\\
\Lambda^2,& \mbox{if } G_n=(S)O_{2n+1},\\
Sym^2,& \mbox{if } G_n=Sp_{2n}\mbox{ or }(S)O_{2n}.
\end{cases}\]
Here ``Asai'' is the Asai representation of the $L$-group of the scalar restriction $Res_{E/F}(GL_k)$, $\chi_{E/F}$ is the quadratic character associated to the field extension $E/F$ via class field theory, and $Sym^2$ (resp. $\Lambda^2$) is the symmetric (resp. exterior) second power of the standard representation of $GL_k$.

Indeed, based on the calculation of Plancherel measure, we know that, up to invertible elements in $\mathbb{C}[q^{-s},~q^s]$,
\[M^*_{c,a}(-s,~\tau^*,~\sigma_{\bar{r}})\circ M^*_{c,a}(s,~\tau,~\sigma_{\bar{r}})=\frac{1}{\beta_{c,a}(s,~\tau,~\sigma_{\bar{r}})\beta_{c,a}(-s,~\tau^*,~\sigma_{\bar{r}})}. \]
Here $\beta_{c,a}(s,~\tau,~\sigma_{\bar{r}})$ is given similarly as $\alpha_{c,a}(s,~\tau,~\sigma_{\bar{r}})$ by
\[\beta_{c,a}(s,~\tau,~\sigma_{\bar{r}}):=\prod_{i=1}^{\lceil\frac{c}{2}\rceil}L(2s+c+2-2i,\tau_a,\rho)\prod_{i=1}^{\lfloor\frac{c}{2}\rfloor}L(2s+c+1-2i,\tau_a,\rho^-)L(s+\frac{c+1}{2},\tau_a\times \sigma_{\bar{r}}). \]
Now we can state our Main Theorem as follows.
\begin{mthm}\label{mainthm}
	Retain the notions as above. We have
	\[M^*_{c,a}(s,~\tau,~\sigma_{\bar{r}}):=\frac{1}{\alpha_{c,a}(s,~\tau,~\sigma_{\bar{r}})}M_{c,a}(s,~\tau,~\sigma_{\bar{r}})\mbox{ is holomorphic for all }s\in \mathbb{C}. \]
\end{mthm}
\begin{rem}
	In the paper, we will prove our Main Theorem for the case $\tau\simeq \tau^*$, leaving the case $\tau\not\simeq \tau^*$ for the interested readers, for the reason that the holomorphy of $M_{c,a}^*(s,~\tau,~\sigma_{\bar{r}})$ for the latter case could be analyzed similarly as done for the former case. Therefore, we always assume that $\tau\simeq \tau^*$ in the remaining part of the paper.
\end{rem}
One may notice that such a problem for degenerate principal series, i.e., when $n_{0}=0$ and $a=1$, has attracted a lot of attention in the development of the theory of theta correspondence and the classical doubling method. It has been investigated and fully understood for all classical groups, unnecessary quasi-split, by Harris--Kudla--Sweet among others (see \cite{piatetskishapirotriple,harriskudlasweet1996,kudlarallis1992,kudlasweet1997,sweet1995,ikeda1992location,yamana2011,yamanatheta}. On the other hand, it is well-known that $M^*_{c,1}(s,~\tau,~1)$ is always non-zero, i.e., the singularity of $M_{c,1}(s,~\tau,~1)$ is exactly given by the normalization factor $\alpha_{c,1}(s,~\tau,~1)$. In view of this and the non-zero result for $M^*_{1,a}(s,~\tau,~\sigma_{\bar{r}})$ proved in \cite{luo2021casselmanshahidiconjecture}, one may wonder if the non-zero property holds in our setting in the paper. Unfortunately, we do not know how to prove it in general, but a special case, which covers the degenerate principal series case, can be deduced from the Plancherel formula of intertwining operators in what follows. For simplicity, we write $M^*_c(s,~\tau,~\sigma_{\bar{r}})$ (resp. $M^*(s,~\tau_a,~\sigma_{\bar{r}})$) for $M^*_{c,1}(s,~\tau,~\sigma_{\bar{r}})$ (resp. $M^*_{1,a}(s,~\tau_a,~\sigma_{\bar{r}})$), and $\alpha_c(s,~\tau,~\sigma_{\bar{r}})$ (resp. $\alpha(s,~\tau_a,~\sigma_{\bar{r}})$) for $\alpha_{c,1}(s,~\tau,~\sigma_{\bar{r}})$ (resp. $\alpha_{1,a}(s,~\tau,~\sigma_{\bar{r}})$), similarly for other related notions. In particular, $\alpha_1(s,~\tau,~\sigma_{\bar{r}})=\alpha(s,~\tau,~\sigma_{\bar{r}})$. Then we have
\begin{cor}\label{rmknonzero}
	$M^*(s,~\tau_a,~\sigma_{\bar{r}})$ and $M^*_c(s,~\tau,~\sigma)$ are always non-zero for $s\in \mathbb{C}$.
\end{cor}  
\begin{proof}
	The former case has been proved in \cite{luo2021casselmanshahidiconjecture}. As for the latter case, it follows easily from the facts that $$M^*_c(-s,~\tau,~\sigma)\circ M^*_c(s,~\tau,~\sigma)=\frac{1}{\beta_c(s,~\tau,~\sigma)\beta_c(-s,~\tau,~\sigma)} \mbox{ (up to invertible elements) }$$ and $M_c(s,~\tau,~\sigma)\neq 0$ (see \cite[IV.1.(10)]{waldspurger2003formule}). To be precise, the fact that $M_c(s,~\tau,~\sigma)\neq 0$ implies that the possible zeros of $ M^*_c(s,~\tau,~\sigma)$ are given by $$\frac{1}{\alpha_c(s,~\tau,~\sigma)},$$ but the relation $$M^*_c(-s,~\tau,~\sigma)\circ M^*_c(s,~\tau,~\sigma)=\frac{1}{\beta_c(s,~\tau,~\sigma)\beta_c(-s,~\tau,~\sigma)}$$ says that the possible zeros of $M^*_c(s,~\tau,~\sigma)$ are given by $$\frac{1}{\beta_c(s,~\tau,~\sigma)\beta_c(-s,~\tau,~\sigma)}.$$ On one hand, one can check readily that the greatest common divisor of $$\frac{1}{\alpha_c(s,~\tau,~\sigma)}$$ and $$\frac{1}{\beta_c(s,~\tau,~\sigma)\beta_c(-s,~\tau,~\sigma)}$$ is given by $$g.c.d.\left(L(s-\frac{c-1}{2},\tau\times\sigma)^{-1},~L(-2s+c-1,\tau,\rho^-)^{-1}\right).$$ Moreover, $$M_c\left(-\frac{c-1}{2},~\tau,~\sigma\right)\circ M_c\left(\frac{c-1}{2},~\tau,~\sigma\right)=id. \mbox{ (up to a non-zero scalar)}.$$
	On the other hand, if $L(0,\tau\times\sigma)=\infty=L(0,\tau,\rho^-)$, then $\tau\rtimes \sigma$ is irreducible. So at $s=\frac{c-1}{2}$, $\rho_c(\tau)|det(\cdot)|^{\frac{c-1}{2}}\rtimes \sigma$ is reducible and is a quotient of the standard module (cf. \cite{lapid2017some}),
	\[\tau|det(\cdot)|^{c-1}\times \tau|det(\cdot)|^{c-2}\times\cdots\times \tau|det(\cdot)|^1\rtimes (\tau\rtimes \sigma).\]
	Thus the reducible representation $\rho_c(\tau)|det(\cdot)|^{\frac{c-1}{2}}\rtimes \sigma$ has a unique irreducible quotient constituent, which in turn implies that the holomorphic intertwining operator $M_c(-\frac{c-1}{2},~\tau,~\sigma)$ maps a constituent of $\rho_c(\tau)|det(\cdot)|^{-\frac{c-1}{2}}\rtimes \sigma$ to zero, i.e., $M_c(s,~\tau,~\sigma)$ restricting to such a constituent in $\rho_c(\tau)|det(\cdot)|^{\frac{c-1}{2}}\rtimes \sigma$ has a simple pole at $s=\frac{c-1}{2}$. Whence $M_c^*(\frac{c-1}{2},~\tau,~\sigma)\neq 0$.
	Thus $M_c^*(s,~\tau,~\sigma)$ is non-zero for all $s\in\mathbb{C}$.
\end{proof}
To record the non-zero expectation of $M^*_{c,a}(s,~\tau,~\sigma_{\bar{r}})$, we state it as a conjecture as follows.
\begin{conj}\label{conj}
	Keep the notation as before. We have
	\[M^*_{c,a}(s,~\tau,~\sigma_{\bar{r}})\mbox{ is alwasys non-zero for any }s\in \mathbb{C}.  \]
\end{conj} 

\section{proof of main theorem \ref{mainthm}: degenerate case gl }\label{caseGL}
Before turning to the proof of our Main Theorem \ref{mainthm}, let us first prove its analogue statement for general linear groups $GL$, which will be needed in what follows. Recall that $\tau$ is a supercuspidal representation of $GL_k$, and $\tau_a$ (resp. $\rho_c(\tau_a)$) is the associated Steinberg (resp. Speh) representation of $GL_{ak}$ (resp. $GL_{akc}$), i.e., 
\[\tau_a\hookrightarrow |det(\cdot)|^{\frac{(a-1)}{2}}\tau\times|det(\cdot)|^{\frac{(a-3)}{2}}\tau\times\cdots\times |det(\cdot)|^{-\frac{(a-1)}{2}}\tau,\] \[(\mbox{resp. } \rho_c(\tau_a)\hookrightarrow |det(\cdot)|^{-\frac{(c-1)}{2}}\tau_a\times|det(\cdot)|^{-\frac{(c-3)}{2}}\tau_a\times\cdots\times |det(\cdot)|^{\frac{(c-1)}{2}}\tau_a).\]
Consider the standard intertwining operator as in \cite[Part 1]{moeglin1989residue}, for $s\in\mathbb{C}$,
\[
M_{GL}(s,~\rho_{c}(\tau_a),~\rho_d(\tau_b)):~ |det(\cdot)|^s\rho_c(\tau)\times |det(\cdot)|^{-s}\rho_d(\tau_b)\longrightarrow |det(\cdot)|^{-s}\rho_d(\tau_b)\times |det(\cdot)|^s\rho_c(\tau_a),
\]
we define the corresponding normalized version to be
\[M^*_{GL}(s,~\rho_{c}(\tau_a),~\rho_d(\tau_b)):=\frac{1}{\alpha_{GL}(s,~\rho_{c}(\tau),~\rho_{d}(\tau))}M_{GL}(s,~\rho_{c}(\tau_a),~\rho_d(\tau_b)). \]
Here $$\alpha_{GL}(s,~\rho_c(\tau_a),~\rho_d(\tau_b)):=\prod_{j=\frac{|c-d|}{2}}^{\frac{c+d-2}{2}}L(2s-j,~\tau_a\times \tau^\vee_b),$$ and $\tau^\vee$ is the contragredient representation of $\tau$. Then we have
\begin{prop}\label{propGL}
	Retain the notions as above. $$M^*_{GL}(s,~\rho_{c}(\tau_a),~\rho_d(\tau_b))$$ is holomorphic for any $s\in\mathbb{C}$.
\end{prop}
The main idea is to analyze the discrepancies of normalization factors $\alpha_{GL}(\cdots)$ corresponding to different reduced decompositions given by different embeddings of our induced data. We record those decompositions in the following for later use.

\underline{Reduced decompositions}: viewing $\rho_{c}(\tau_a)$ or $\rho_{d}(\tau_b)$ as a subrepresentation, i.e., if at least one of $a,~b,~c,$ and $d$ is greater than 1,
\begin{align*}
\mbox{Way 1: }&\rho_{c}(\tau_a)\hookrightarrow |det(\cdot)|^{-\frac{1}{2}}\rho_{c-1}(\tau_{a})\times |det(\cdot)|^{\frac{c-1}{2}}\tau_a,\\
\mbox{Way 2: }&\rho_{c}(\tau_{a})\hookrightarrow |det(\cdot)|^{-\frac{(c-1)}{2}}\tau_a\times |det(\cdot)|^{\frac{1}{2}}\rho_{c-1}(\tau_{a}),\\
\mbox{Way 3: }&\rho_{d}(\tau_b)\hookrightarrow |det(\cdot)|^{-\frac{1}{2}}\rho_{d-1}(\tau_{b})\times |det(\cdot)|^{\frac{d-1}{2}}\tau_b,\\
\mbox{Way 4: }&\rho_{d}(\tau_b)\hookrightarrow |det(\cdot)|^{-\frac{(d-1)}{2}}\tau_b\times |det(\cdot)|^{\frac{1}{2}}\rho_{d-1}(\tau_{b}).
\end{align*}
Notice that $\rho_c(\tau_a)$ can also be viewed as a unique subrepresentation in the following way (see \cite{moeglinwaldspurger1986involution})
\[\rho_c(\tau_a)\hookrightarrow \rho_c(\tau)|det(\cdot)|^\frac{a-1}{2}\times\cdots\times\rho_c(\tau)|det(\cdot)|^{-\frac{a-1}{2}}. \]
Similarly we have
\begin{align*}
\mbox{Way 1': }&\rho_{c}(\tau_a)\hookrightarrow |det(\cdot)|^{\frac{1}{2}}\rho_{c}(\tau_{a-1})\times |det(\cdot)|^{-\frac{a-1}{2}}\rho_{c}(\tau),\\
\mbox{Way 2': }&\rho_{c}(\tau_{a})\hookrightarrow |det(\cdot)|^{\frac{(a-1)}{2}}\rho_{c}(\tau)\times |det(\cdot)|^{-\frac{1}{2}}\rho_{c}(\tau_{a-1}),\\
\mbox{Way 3': }&\rho_{d}(\tau_b)\hookrightarrow |det(\cdot)|^{\frac{1}{2}}\rho_{d}(\tau_{b-1})\times |det(\cdot)|^{-\frac{b-1}{2}}\rho_{d}(\tau),\\
\mbox{Way 4': }&\rho_{d}(\tau_b)\hookrightarrow |det(\cdot)|^{\frac{(b-1)}{2}}\rho_{d}(\tau)\times |det(\cdot)|^{-\frac{1}{2}}\rho_{d}(\tau_{b-1}).
\end{align*}
One may notice that the above definitions could be merged into two classes, instead of four classes, we would like to keep it there to distinguish their different places in our inducing data, as one can see that the terminologies Way 3 and Way 4 will also be used to describe reduced decompositions with respect to $\sigma_{\bar{r}}$ in the setting of classical groups later on.

To start the proof of Proposition \ref{propGL}, we first prove two basic cases, Case $a=b=1$ and Case $c=d=1$, which will involve the reduction principles and the induction argument needed in the general case, as follows. 
\begin{lem}\label{lemGL}
	$M^*_{GL}(s,~\tau_a,~\tau_b)$ and $M^*_{GL}(s,~\rho_{c}(\tau),~\rho_{d}(\tau))$ are holomorphic for any $s\in \mathbb{C}$.
\end{lem}
\begin{proof}
	\underline{Case $c=d=1$}: One may notice that this has been proved in \cite{luo2021casselmanshahidiconjecture}, but we still would like to provide a simplified argument for completeness. Consider the following commutative diagrams given by the embeddings $\tau_a\hookrightarrow|det(\cdot)|^{\frac{1}{2}}\tau_{a-1}\times |det(\cdot)|^{-\frac{a-1}{2}}\tau$ and $\tau_b\hookrightarrow|det(\cdot)|^{\frac{b-1}{2}}\tau\times|det(\cdot)|^{-\frac{1}{2}}\tau_{b-1}$ respectively, 
	
	\underline{Way 1':}
	\[\xymatrix{|det(\cdot)|^{s}\tau_a\times |det(\cdot)|^{-s}\tau_b  \ar@{^(->}[r] \ar[dd]_{M_{GL}(s,\tau_a,\tau_b)} & |det(\cdot)|^{s+\frac{1}{2}}\tau_{a-1}\times\mbox{$\underbrace{|det(\cdot)|^{s-\frac{a-1}{2}}\tau\times |det(\cdot)|^{-s}\tau_b}$}\ar[d]^{M_{GL}(\cdots)}\\
		&\mbox{$\underbrace{|det(\cdot)|^{s+\frac{1}{2}}\tau_{a-1}\times |det(\cdot)|^{-s}\tau_b}$}\times|det(\cdot)|^{s-\frac{a-1}{2}}\tau \ar[d]^{M_{GL}(\cdots)}\\ 
		|det(\cdot)|^{-s}\tau_b\times |det(\cdot)|^{s}\tau_a \ar@{^(->}[r]&|det(\cdot)|^{-s}\tau_b\times|det(\cdot)|^{s+\frac{1}{2}}\tau_{a-1}\times|det(\cdot)|^{s-\frac{a-1}{2}}\tau,
	}\]
	
	\underline{Way 4':}
	\[\xymatrix{|det(\cdot)|^{s}\tau_a\times |det(\cdot)|^{-s}\tau_b  \ar@{^(->}[r] \ar[dd]_{M_{GL}(s,\tau_a,\tau_b)} & \mbox{$\underbrace{|det(\cdot)|^{s}\tau_{a}\times|det(\cdot)|^{-s+\frac{b-1}{2}}\tau}$}\times |det(\cdot)|^{-s-\frac{1}{2}}\tau_{b-1}\ar[d]^{M_{GL}(\cdots)}\\
		&|det(\cdot)|^{-s+\frac{b-1}{2}}\tau\times\mbox{$\underbrace{|det(\cdot)|^{s}\tau_{a}\times|det(\cdot)|^{-s-\frac{1}{2}}\tau_{b-1}}$} \ar[d]^{M_{GL}(\cdots)}\\ 
		|det(\cdot)|^{-s}\tau_b\times |det(\cdot)|^{s}\tau_a \ar@{^(->}[r] & |det(\cdot)|^{-s+\frac{b-1}{2}}\tau\times|det(\cdot)|^{-s-\frac{1}{2}}\tau_{b-1}\times|det(\cdot)|^{s}\tau_a.
	}\]
	It is an easy calculation to see that 
	\[L(2s,\tau_a\times\tau_b)=\begin{cases}
	L(2s-\frac{b-1}{2},\tau_a\times \tau)L(2s+\frac{1}{2},\tau_a\times \tau_{b-1}), & \mbox{ if }a\geq b;\\
	L(2s-\frac{a-1}{2},\tau\times \tau_b)L(2s+\frac{1}{2},\tau_{a-1}\times \tau_b), & \mbox{ if }a< b.
	\end{cases} \]
	Which implies that the normalization factors of intertwining operators also match each other under Way 4' if $a\geq b$ or Way 1' if $a<b$. Thus it reduces to prove the holomorphy of $M^*_{GL}(s,~\tau_a,~\tau)$ and $M^*_{GL}(s,~\tau,~\tau_b)$. In what follows, we only discuss the former case, while the latter case can be proved similarly. Consider the reduced decomposition given by the above Way 1' with $a>1$ and $b=1$, and compute the discrepancy $P_{1'}$ of the normalization factors on the left hand side and the right hand side corresponding to Way 1', we obtain
	\[P_{1'}:=\frac{\alpha_{GL}(\cdots)\alpha_{GL}(\cdots)}{\alpha_{GL}(s,~\tau_a,~\tau)}=
	L\left(2s-\frac{a-1}{2}, \tau\times \tau\right). \]
	Therefore by induction on $a$, we know that $M^*_{GL}(s,~\tau_a,~\tau)$ is holomorphic except at the poles of $P_{1'}$, which is at $2s=\frac{a-1}{2}>0$. On the other hand, it is well-known that $M_{GL}(s,~\tau_a,~\tau)$ is well-defined at $Re(s)>0$, thus $M^*_{GL}(s,~\tau_a,~\tau)$ is holomorphic for any $s\in\mathbb{C}$ following from the fact that $L(2s,\tau_a\times \tau)$ has poles only at $Re(s)\leq 0$. 
	
	\underline{Case $a=b=1$}: Consider the following commutative diagrams corresponding to the embeddings $\rho_{c}(\tau)\hookrightarrow|det(\cdot)|^{-\frac{1}{2}}\rho_{c-1}(\tau)\times |det(\cdot)|^{\frac{c-1}{2}}\tau$ and $\rho_d(\tau)\hookrightarrow|det(\cdot)|^{-\frac{d-1}{2}}\tau\times|det(\cdot)|^{\frac{1}{2}}\rho_{d-1}(\tau)$ respectively, 
	
	\underline{Way 1:}
	\[\xymatrix{|det(\cdot)|^{s}\rho_{c}(\tau)\times |det(\cdot)|^{-s}\rho_{d}(\tau)  \ar@{^(->}[r] \ar[dd]_{M_{GL}(s,\rho_{c}(\tau),\rho_d(\tau))} & |det(\cdot)|^{s-\frac{1}{2}}\rho_{c-1}(\tau)\times\mbox{$\underbrace{|det(\cdot)|^{s+\frac{c-1}{2}}\tau\times |det(\cdot)|^{-s}\rho_d(\tau)}$}\ar[d]^{M_{GL}(\cdots)}\\
		&\mbox{$\underbrace{|det(\cdot)|^{s-\frac{1}{2}}\rho_{c-1}(\tau)\times |det(\cdot)|^{-s}\rho_d(\tau)}$}\times|det(\cdot)|^{s+\frac{c-1}{2}}\tau \ar[d]^{M_{GL}(\cdots)}\\ 
		|det(\cdot)|^{-s}\rho_d(\tau)\times |det(\cdot)|^{s}\rho_c(\tau) \ar@{^(->}[r]&|det(\cdot)|^{-s}\rho_d(\tau)\times|det(\cdot)|^{s-\frac{1}{2}}\rho_{c-1}(\tau)\times|det(\cdot)|^{s+\frac{c-1}{2}}\tau,
	}\]
	
	\underline{Way 4:}
	\[\xymatrix{|det(\cdot)|^{s}\rho_c(\tau)\times |det(\cdot)|^{-s}\rho_d(\tau)  \ar@{^(->}[r] \ar[dd]_{M_{GL}(s,\rho_{c}(\tau),\rho_{d}(\tau))} & \mbox{$\underbrace{|det(\cdot)|^{s}\rho_{c}(\tau)\times|det(\cdot)|^{-s-\frac{d-1}{2}}\tau}$}\times |det(\cdot)|^{-s+\frac{1}{2}}\rho_{d-1}(\tau)\ar[d]^{M_{GL}(\cdots)}\\
		&|det(\cdot)|^{-s-\frac{d-1}{2}}\tau\times\mbox{$\underbrace{|det(\cdot)|^{s}\rho_{c}(\tau)\times|det(\cdot)|^{-s+\frac{1}{2}}\rho_{d-1}(\tau)}$} \ar[d]^{M_{GL}(\cdots)}\\ 
		|det(\cdot)|^{-s}\rho_d(\tau)\times |det(\cdot)|^{s}\rho_c(\tau) \ar@{^(->}[r] & |det(\cdot)|^{-s-\frac{d-1}{2}}\tau\times|det(\cdot)|^{-s+\frac{1}{2}}\rho_{d-1}(\tau)\times|det(\cdot)|^{s}\rho_c(\tau).
	}\]
	It is an easy calculation to see that 
	\[\alpha_{GL}(s,~\rho_c(\tau),~\rho_d(\tau))=\begin{cases}
	\alpha_{GL}(s+\frac{d-1}{4},~\rho_{c}(\tau),~\tau)\cdot\alpha_{GL}(s-\frac{1}{4},~\rho_{c}(\tau),~\rho_{d-1}(\tau)), & \mbox{ if }c\geq d;\\
	\alpha_{GL}(s+\frac{c-1}{4},~\tau,~ \rho_d(\tau))\cdot\alpha_{GL}(s-\frac{1}{4},~\rho_{c-1}(\tau),~ \rho_d(\tau)), & \mbox{ if }c< d.
	\end{cases} \]
	Which implies that the normalization factors of intertwining operators also match each other under Way 4 if $c\geq d$ or Way 1 if $c<d$. Thus it reduces to prove the holomorphy of $M^*_{GL}(s,~\rho_c(\tau),~\tau)$ and $M^*_{GL}(s,~\tau,~\rho_{d}(\tau))$. In what follows, we only discuss the former case, while the latter case can be proved similarly. Consider the reduced decompositions given by the above Way 1 and the following Way 2 corresponding to the embedding $\rho_{c}(\tau)\hookrightarrow|det(\cdot)|^{-\frac{c-1}{2}}\tau\times |det(\cdot)|^{\frac{1}{2}}\rho_{c-1}(\tau)$, with $c>1$ and $d=1$,
	
	\underline{Way 2:}
	\[\xymatrix{|det(\cdot)|^{s}\rho_{c}(\tau)\times |det(\cdot)|^{-s}\rho_{d}(\tau)  \ar@{^(->}[r] \ar[dd]_{M_{GL}(s,\rho_{c}(\tau),\rho_{d}(\tau))} & |det(\cdot)|^{s-\frac{c-1}{2}}\tau\times\mbox{$\underbrace{|det(\cdot)|^{s+\frac{1}{2}}\rho_{c-1}(\tau)\times |det(\cdot)|^{-s}\rho_d(\tau)}$}\ar[d]^{M_{GL}(\cdots)}\\
		&\mbox{$\underbrace{|det(\cdot)|^{s-\frac{c-1}{2}}\tau\times |det(\cdot)|^{-s}\rho_{d}(\tau)}$}\times|det(\cdot)|^{s+\frac{1}{2}}\rho_{c-1}(\tau) \ar[d]^{M_{GL}(\cdots)}\\ 
		|det(\cdot)|^{-s}\rho_{d}(\tau)\times |det(\cdot)|^{s}\rho_c(\tau) \ar@{^(->}[r]&|det(\cdot)|^{-s}\rho_{d}(\tau)\times|det(\cdot)|^{s-\frac{c-1}{2}}\tau\times|det(\cdot)|^{s+\frac{1}{2}}\rho_{c-1}(\tau),
	}\] 
	and compute the discrepancies $P_i$ of the normalized factors on the left hand side and the right hand side corresponding to Way i, i=1, 2, we obtain
	\[P_i:=\frac{\alpha_{GL}(\cdots)\alpha_{GL}(\cdots)}{\alpha_{GL}(s,~\rho_{c}(\tau),~\tau)}=\begin{cases}
	L\left(2s+\frac{c-1}{2},\tau\times \tau\right),&\mbox{ if }i=1;\\
	L\left(2s-\frac{c-3}{2}, \tau\times \tau\right),&\mbox{ if }i=2.
	\end{cases} \]
	Therefore we know that the set of common poles of $P_1$ and $P_2$ is non-empty, i.e., $(P_1^{-1},~P_2^{-1})\neq 1$, if and only if 
	\[\frac{c-3}{2}=-\frac{c-1}{2},\mbox{ i.e., }c=2. \]
	Which in turn says that the common pole is at $2s=-\frac{1}{2}$. Now we show directly that $M^*_{GL}(s,~\rho_2(\tau),~\tau)$ is holomorphic at $2s=-\frac{1}{2}$. This follows from the detailed information we have in the commutative diagram given by Way 2 as follows.
	
	\[\xymatrix{|det(\cdot)|^{-\frac{1}{2}}\rho_2(\tau)\times \tau \ar@{^(->}[r] \ar[dd]_{M_{GL}(s,\rho_2(\tau),\tau)} & |det(\cdot)|^{-1}\tau\times\mbox{$\underbrace{\tau\times \tau}$}\ar[d]^{\mbox{simple pole}}_{scalar}\\
		&\mbox{$\underbrace{|det(\cdot)|^{-1}\tau\times \tau}$}\times\tau \ar[d]_{subrep.\mapsto 0}^{\mbox{simple zero}}\\ 
		\tau\times |det(\cdot)|^{-\frac{1}{2}}\rho_2(\tau) \ar@{^(->}[r]&\tau\times|det(\cdot)|^{-1}\tau\times\tau,
	}\] 
	As the top intertwining operator on the right hand side is a scalar of simple pole, and the bottom one restricting to $|det(\cdot)|^{-\frac{1}{2}}\rho_2(\tau)\times \tau$ is a zero map of simple order, so the composition of those two operators $M_{GL}(s,~\rho_2(\tau),~\tau)$ is holomorphic at $2s=-\frac{1}{2}$, thus $M^*_{GL}(s,~\rho_{c}(\tau),~\tau)$ is holomorphic in $s$ by induction. Whence our Lemma holds.
\end{proof}
Now return to our proof of Proposition \ref{propGL}.
\begin{proof}[Proof of Proposition \ref{propGL}]
	This follows easily from the reduction argument used in the proof of Lemma \ref{lemGL}. To be precise, if $c<d$, we apply the reduced decomposition Way 1, i.e., \[\rho_{c}(\tau_a)\hookrightarrow|det(\cdot)|^{-\frac{1}{2}}\rho_{c-1}(\tau_a)\times |det(\cdot)|^{\frac{c-1}{2}}\tau_a,\]
	and if $c\geq d$, use Way 4, i.e., \[\rho_d(\tau_b)\hookrightarrow|det(\cdot)|^{-\frac{d-1}{2}}\tau_b\times|det(\cdot)|^{\frac{1}{2}}\rho_{d-1}(\tau_b),\]
	a similar argument as in Lemma \ref{lemGL} for the case $a=b=1$ shows that it suffices to consider the case $c=1$ and the case $d=1$. On the other hand, if $a<b$, we input the reduced decomposition Way 1', i.e., \[\rho_{c}(\tau_a)\hookrightarrow|det(\cdot)|^{\frac{1}{2}}\rho_{c}(\tau_{a-1})\times |det(\cdot)|^{-\frac{a-1}{2}}\rho_{c}(\tau),\]
	and if $a\geq b$, apply Way 4', i.e., \[\rho_d(\tau_b)\hookrightarrow|det(\cdot)|^{\frac{b-1}{2}}\rho_{d}(\tau)\times|det(\cdot)|^{-\frac{1}{2}}\rho_d(\tau_{b-1}).\]
	A similar argument as in Lemma \ref{lemGL} for the case $c=d=1$ implies that it suffices to consider the case $a=1$ and the case $b=1$. Combining those two reduction steps together, it suffices to consider the following cases (other analogue cases could be handled in a similar way and have been omitted), 
	\[|det(\cdot)|^{s}\tau\times |det(\cdot)|^{-s}\rho_{d}(\tau_b),~ \mbox{and }|det(\cdot)|^{s}\rho_{c}(\tau)\times |det(\cdot)|^{-s}\tau_b.\]
	Now we start to do a further reduction step so that we can apply our results in Lemma \ref{lemGL} directly. For Case $|det(\cdot)|^{s}\rho_{c}(\tau)\times |det(\cdot)|^{-s}\tau_b$ (resp. Case $|det(\cdot)|^{s}\tau\times |det(\cdot)|^{-s}\rho_{d}(\tau_b)$), one can check readily that it reduces to consider the case $b=1$ (resp. Case $b=1$ and Case $d=1$) by comparing the reduced decompositions Way 1 and Way 4' (resp. Way 4 and Way 4') which tells us that $(P_1^{-1},~P_{4'}^{-1})=1$ (resp. $(P_4^{-1},~P_{4'}^{-1})=1$). Thus we complete the proof by Lemma \ref{lemGL}. 
\end{proof}	
\section{proof of main theorem \ref{mainthm}: reduction steps}\label{redsteps}
Now back to our setup of classical groups. Recall that $G_{n}$ is a quasi-split classical group of a fixed type with $n=akc+n_0$, $\tau$ (resp. $\sigma_{\bar{r}}$) is a supercupidal (resp. generic discrete series) representation of $GL_k$ (resp. $G_{n_0}$), we have the associated induced representation $|det(\cdot)|^s\rho_c(\tau_a)\rtimes \sigma_{\bar{r}}$ of $G_n$ and the normalized intertwining operator $$M^*_{c,a}(s,~\tau,~\sigma_{\bar{r}}):=\frac{1}{a_{c,a}(s,~\tau,~\sigma_{\bar{r}})}M_{c,a}(s,~\tau,~\sigma_{\bar{r}})$$ with 
\[M_{c,a}(s,\tau,\sigma_{\bar{r}}):~|det(\cdot)|^s\rho_c(\tau)\rtimes \sigma_{\bar{r}}\longrightarrow |det(\cdot)|^{-s}\rho_c(\tau)\rtimes \sigma_{\bar{r}} \]
the standard intertwining operator with respect to the Weyl element
\[w_{c;a}:=(-1)^{akc}\begin{pmatrix}
&  & I_{akc} \\ 
& I_{N_0} &  \\ 
\pm I_{akc} &  & 
\end{pmatrix} . \]
Note that we have the following reduced decompositions of $w_{c;a}$ via an easy calculation
\[w_{c;a}=w_{c-j;a}w_{c-j,j;a}w_{j;a}, \]
where $j=1,~\cdots,~c-1$ and
\[w_{c-j,j;a}:=\begin{pmatrix}
& I_{akj} &  &  &  \\ 
I_{ak(c-j)} &  &  &  &  \\ 
&  & I_{N_0} &  &  \\ 
&  &  &  & I_{ak(c-j)} \\ 
&  &  & I_{akj} & 
\end{pmatrix}. \]
Similarly, we have reduced decompositions with respect to partitions of $a$.
In view of our Proposition \ref{propGL} and the inductive structure of $M^*_{c,a}(s,~\tau,~\sigma_{\bar{r}})$, it is quite natural to prove the holomorphy of $M^*_{c,a}(s,~\tau,~\sigma_{\bar{r}})$ via induction on $(c,~a)$ and a reduction argument on $\sigma_{\bar{r}}$. The only issue is the discrepancy between the normalization factors, i.e., $$\alpha_{c,a}(\cdots)\neq \alpha_{c-j,a}(\cdots)\alpha_{GL}(\cdots)\alpha_{j,a}(\cdots).$$
Regarding this, for $j=1$ and $c-1$, we define the corresponding discrepancies by
\begin{align*}
P_1:&=\frac{\alpha_{c-1,a}(s-\frac{1}{2},~\tau,~\sigma_{\bar{r}})\cdot \alpha_{GL}(2s-\frac{1}{2}+\frac{c-1}{2},~\tau)\cdot \alpha_{1,a}(s+\frac{c-1}{2},~\tau,~\sigma_{\bar{r}})}{ \alpha_{c,a}(s,~\tau,~\sigma_{\bar{r}})}\\
P_{2}:&=\frac{\alpha_{1,a}(s-\frac{c-1}{2},~\tau,~\sigma_{\bar{r}})\cdot \alpha_{GL}(2s+\frac{1}{2}-\frac{c-1}{2},~\tau)\cdot \alpha_{c-1,a}(s+\frac{1}{2},~\tau,~\sigma_{\bar{r}})}{ \alpha_{c,a}(s,~\tau,~\sigma_{\bar{r}})}.
\end{align*}
Similarly we have notions $P_{1'}$ and $P_{2'}$ (resp. $P_3$ and $P_4$) for reduced decompositions with respect to $a$ (resp. $\sigma_{\bar{r}}$). It turns out that the issue mentioned above will be a big advantage for the holomorphicity problem we need to handle. As seen from the argument in Proposition \ref{propGL}, the location of poles of discrepancies $P_1$, $P_2$, $P_3$, and $P_4$ etc. are quite different, and this intrinsic non-symmetry property is our main idea hidden in the argument. Now let us begin our simple proof of Main Theorem \ref{mainthm} via the non-symmetry property of normalization factors of reduced decompositions observed and applied earlier in what follows. 

We first recall the classification of generic discrete series of $G_n$ and the associated Langlands parameters (see \cite{tadic1996generic,muic1998,moeglin2002classification,moeglin2002construction,cogdellkimshapiroshahidi2004,jantzen2000squareintegrable,jantzen2000squareintegrableII,jiangsoudry2003,jiangsoudry2012,arthur2013endoscopic,jantzenliu2014}). The Langlands parameter part is only needed to see clearly the decomposition formula for the tensor product $L$-function $L(s,\tau_a\times \sigma_{\bar{r}})$ in what follows. Indeed, it could be deduced from Langlands--Shahidi's theory, especially the multiplicativity of $\gamma$-factors (see \cite{shahidi1990multiplicativity}). 
\begin{enumerate}[(i).]
	\item generic supercuspidal $\sigma$: $N=2n+1$ or $2n$ depends only on $G_n$,
	\[\phi_\sigma:~W_F\longrightarrow  {^{L}G_n}(\mathbb{C})\stackrel{Std}{\longrightarrow} GL_N(\mathbb{C}) \mbox{ satisfies }\]
	\begin{itemize}
		\item $\phi_\sigma=\oplus_i \phi_{\rho_i}$ with $\rho_i:~ W_F\longrightarrow GL_{N_i}(\mathbb{C})$ irreducible associated to $\rho_i$ supercuspidal representation of $GL_{N_i}$ of type ${^{L}G_n}$,
		\item $\phi_{\rho_i}\not\simeq \phi_{\rho_j}$ for any $i\neq j$.
	\end{itemize}
	\item generic discrete series $\sigma_{\bar{r}}$ supported on $\tau$ and $\sigma$: 
	\[\phi_{\sigma_{\bar{r}}}:~W_F\longrightarrow  {^{L}G_n}(\mathbb{C})\stackrel{Std}{\longrightarrow} GL_N(\mathbb{C}) \mbox{ satisfies } \]
	\begin{itemize}
		\item $\phi_{\sigma_{\bar{r}}}=\phi_{\tau}\otimes S_{r_1}\oplus\phi_{\tau}\otimes S_{r_2}\oplus\cdots\oplus\phi_{\tau}\otimes S_{r_t}\oplus\phi_\sigma$,
		\item $r_1>r_2>\cdots>r_t\geq -1$ of the same parity and $t$ is even,
		\item set $\phi_{\tau}\otimes S_0=0$ and $\phi_{\tau}\otimes S_{-1}=-\phi_{\tau}\otimes S_{1}$.		
	\end{itemize}
	\item generic discrete series $\sigma_{\bar{r}}$ supported on a family of $\tau^{(i)}$ and $\sigma$:
	\[\phi_{\sigma_{\bar{r}}}:~W_F\longrightarrow  {^{L}G_n}(\mathbb{C})\stackrel{Std}{\longrightarrow} GL_N(\mathbb{C}) \mbox{ satisfies } \]
	\begin{itemize}
		\item $\tau^{(i)}\not \simeq \tau^{(j)}$ for any $i\neq j$ are (conjugate) self-dual unitary supercuspidal representations,
		\item $\phi_{\sigma_{r}}=\oplus'_i\phi_{\sigma_{\bar{r}^{(i)}}}$, where $\phi_{\sigma_{\bar{r}^{(i)}}}$ is the Langlands parameter associated to $\sigma_{\bar{r}^{(i)}}$ supported on $\tau^{(i)}$ and $\sigma$, and $\oplus'_i$ means only one $\phi_\sigma$ appearing in the family $\{\phi_{\sigma_{\bar{r}^{(i)}}}\}$ will be summed.
	\end{itemize}
\end{enumerate} 
For simplicity, we consider only the case $\sigma_{\bar{r}}$ supported on $\tau$ and $\sigma$ in the following, while the general case can be stated similarly as it obeys a ``product formula'' rule in the sense of Jantzen (see \cite{jantzen1997supports,jantzen2005duality,jantzenluo}), then the local Langlands correspondence says that 
\[\mbox{generic discrete series $\sigma_{\bar{r}}$ are parameterized by those Langlands parameters }\phi_{\sigma_{\bar{r}}}, i.e., \phi_{\sigma_{\bar{r}}}\leftrightarrow \sigma_{\bar{r}} \mbox{ with}\]
\[  \sigma_{\bar{r}}\hookrightarrow |det(\cdot)|^{\frac{r_1-r_2}{4}}\tau_{\frac{r_1+r_2}{2}}\times |det(\cdot)|^{\frac{r_3-r_4}{4}}\tau_{\frac{r_3+r_4}{2}}\times\cdots\times|det(\cdot)|^{\frac{r_{t-1}-r_t}{4}}\tau_{\frac{r_{t-1}+r_t}{2}}\rtimes\sigma.\tag{$\star$} \]
Meanwhile we have the following formula for the tensor product $L$-function
\[L(s,\tau_a\times\sigma_{\bar{r}})=\prod_{i=1}^{t}L(s,\tau_a\times \tau_{r_i})\cdot L(s,\tau_a\times \sigma)=\prod_{i=1}^{t}L(s,\tau_a\times \tau_{r_i})\cdot L\left(s+\frac{a-1}{2},\tau\times \sigma\right), \tag{$\star\star$}\]
where we set $L(s,\tau_a\times \tau_{0}):=1$ and $L(s,\tau_a\times \tau_{-1}):=L(s,\tau_a\times \tau)^{-1}$.

For the convenience of the readers, we summarize some formulas which are applied often for the calculations carried out in the proof of Main Theorem \ref{mainthm} later on.
\[\begin{cases}
L(2s,\tau_a,\rho)=\prod\limits_{i=1}^{\lceil\frac{a}{2}\rceil}L(2s+a+1-2i,\tau,\rho)\prod\limits_{i=1}^{\lfloor\frac{a}{2}\rfloor}L(2s+a-2i,\tau,\rho^-),&\\
L(2(s-\frac{1}{2}),\tau_{a-1},\rho)=\prod\limits_{i=2}^{\lceil\frac{a+1}{2}\rceil}L(2s+a+1-2i,\tau,\rho)\prod\limits_{i=2}^{\lfloor\frac{a+1}{2}\rfloor}L(2s+a-2i,\tau,\rho^-),&\\
L(2(s+\frac{1}{2}),\tau_{a-1},\rho)=\prod\limits_{i=1}^{\lceil\frac{a-1}{2}\rceil}L(2s+a+1-2i,\tau,\rho)\prod\limits_{i=1}^{\lfloor\frac{a-1}{2}\rfloor}L(2s+a-2i,\tau,\rho^-).&
\end{cases} \tag{A}\]
\[ L(s,\tau_a\times \tau_{r})=\begin{cases}
\prod\limits_{i=-\frac{r-1}{2}}^{\frac{r-1}{2}}L(s+\frac{a-1}{2}+i,\tau\times \tau),&\mbox{ if }a\geq r;\\
\prod\limits_{i=-\frac{a-1}{2}}^{\frac{a-1}{2}}L(s+\frac{r-1}{2}+i,\tau\times \tau),&\mbox{ if }a< r.
\end{cases} \tag{B}\] 
Now back to our proof of Main Theorem \ref{mainthm}. In what follows, we first carry out a reduction argument to reduce the proof of the general case to the special case that $\sigma_{\bar{r}}$ is supported on $\tau$ and $\sigma$, then do a further reduction step to the case that $\sigma_{r}$ is characterized by \[\sigma_r\leftrightarrow \phi_{\sigma_{r}}=\phi_{\tau}\otimes S_{r_1}\oplus\phi_{\tau}\otimes S_{r_2}\oplus\phi_\sigma.\]
The last reduction step involving an induction argument is to reduce the proof to two basic cases, i.e., Case $\rho_{c}(\tau)\rtimes \sigma_{r}$ and Case $\tau_a\rtimes \sigma_r$ with $r_1>a>r_2$. 
\begin{proof}[Proof of Main Theorem \ref{mainthm} (Reduction steps)]
	The three reduction steps are given as follows.
	
	\underline{\bf Step 1 } We reduce the general case $\sigma_{\bar{r}}$ supported on a family of $\tau^{(i)}$ and $\sigma$ to the case $\sigma_{\bar{r}}$ supported on $\tau$ and $\sigma$, i.e.,
	\[\sigma_{\bar{r}}\leftrightarrow\phi_{\sigma_{\bar{r}}}=\phi_{\tau}\otimes S_{r_1}\oplus\phi_{\tau}\otimes S_{r_2}\oplus\cdots\oplus\phi_{\tau}\otimes S_{r_t}\oplus\phi_\sigma. \]
	If $\tau^{(i)}\not\simeq  \tau$ for some $i$, write $\phi_{\sigma_{\bar{r}^{(i)}}}$ a summand in $\phi_{\sigma_{\bar{r}}}=\oplus'_j\phi_{\sigma_{\bar{r}^{(j)}}}$ to be $$\phi_{\sigma_{\bar{r}^{(i)}}}=\phi_{\tau^{(i)}}\otimes S_{r_1^{(i)}}\oplus\phi_{\tau^{(i)}}\otimes S_{r_2^{(i)}}\oplus\cdots\oplus\phi_{\tau^{(i)}}\otimes S_{r_t^{(i)}}\oplus\phi_\sigma,$$
	and denote $\phi_{\sigma_{\bar{r'}}}:=\oplus'_{j\neq i}\phi_{\sigma_{\bar{r}^{(j)}}}$ and 
	\[\phi_{GL}:=\phi_{\tau^{(i)}}\otimes S_{r_1^{(i)}}\oplus\phi_{\tau^{(i)}}\otimes S_{r_2^{(i)}}\oplus\cdots\oplus\phi_{\tau^{(i)}}\otimes S_{r_t^{(i)}}. \]
	Then the reduction step follows from the following facts.
	\begin{enumerate}[(i).]
		\item For $r_1^{(i)}>r_2^{(i)}>\cdots>r_t^{(i)}\geq -1$ of the same parity and $t$ is even, the induced representation 
		\[\pi_{GL}:=|det(\cdot)|^{\frac{r_1^{(i)}-r_2^{(i)}}{4}}\tau_{\frac{r_1^{(i)}+r_2^{(i)}}{2}}\times |det(\cdot)|^{\frac{r_3^{(i)}-r_4^{(i)}}{4}}\tau_{\frac{r_3^{(i)}+r_4^{(i)}}{2}}\times\cdots\times|det(\cdot)|^{\frac{r_{t-1}^{(i)}-r_t^{(i)}}{4}}\tau_{\frac{r_{t-1}^{(i)}+r_t^{(i)}}{2}} \]
		is irreducible (see \cite{bernstein1977induced,zelevinsky1980induced}).
		\item $\sigma_{\bar{r}}\hookrightarrow \pi_{GL}\rtimes \sigma_{\bar{r'}}$ with $\sigma_{\bar{r'}}$ the generic discrete series corresponding to the parameter $\phi_{\sigma_{\bar{r'}}}$.
	\end{enumerate}
	To be precise, consider the following commutative diagram given by the embedding
	
	\underline{Way 3: $\sigma_{\bar{r}}\hookrightarrow \pi_{GL}\rtimes \sigma_{\bar{r'}}$}:
	\[\xymatrix{|det(\cdot)|^{s}\rho_{c}(\tau_a)\rtimes \sigma_{\bar{r}} \ar@{^(->}[r] \ar[ddd]_{M_{c,a}(s,\tau,\sigma_{\bar{r}})} & \mbox{$\underbrace{|det(\cdot)|^{s}\rho_{c}(\tau_a)\times\pi_{GL}}$}\rtimes \sigma_{\bar{r'}}\ar[d]^{M_{GL}(\cdots)}\\
		&\pi_{GL}\times\mbox{$\underbrace{|det(\cdot)|^{s}\rho_{c}(\tau_a)\rtimes\sigma_{\bar{r'}}}$} \ar[d]^{M_{c,a}(s,\tau,\sigma_{\bar{r'}})}\\ 
		&\mbox{$\underbrace{\pi_{GL}\times|det(\cdot)|^{-s}\rho_{c}(\tau_a)}$}\rtimes \sigma_{\bar{r'}}  \ar[d]^{M_{GL}(\cdots)}\\
		|det(\cdot)|^{-s}\rho_{c}(\tau_a)\rtimes \sigma_{\bar{r}} \ar@{^(->}[r] & |det(\cdot)|^{-s}\rho_{c}(\tau_a)\times\pi_{GL}\rtimes\sigma_{\bar{r'}}.
	}\]
	As $\tau^{(i)}\not\simeq \tau$, by a simple analysis of the reduced decomposition in terms of co-rank one intertwining operators, it is well-known that those two intertwining operators $M_{GL}$ on the right hand side are always a non-zero scalar (cf. \cite{silberger1980special}), thus our reduction step holds following from an easy calculation of normalization factors which math each other.
	
	\underline{\bf Step 2 } We reduce the general case $\sigma_{\bar{r}}$ supported on $\tau$ and $\sigma$ to the case $\sigma_r$ associated to one segment, i.e.,
	\[\sigma_{r}\leftrightarrow \phi_{\sigma_{r}}=\phi_{\tau}\otimes S_{r_1}\oplus \phi_{\tau}\otimes S_{r_2}\oplus\phi_\sigma\mbox{ with }r_1>a>r_2. \]
	This follows from the following facts.
	\begin{enumerate}[(a).]
		\item For $a\geq r_1>r_2$ or $r_1>r_2\geq a$, we have the identity of $L$-functions.
		\[L\left(s-\frac{r_1-r_2}{4},\tau_a\times \tau_{\frac{r_1+r_2}{2}}\right)L\left(s+\frac{r_1+r_2}{4},\tau_a\times\tau_{\frac{r_1+r_2}{2}}\right)=L(s,\tau_a\times \tau_{r_1})L(s,\tau_a\times\tau_{r_2}).\]
		\item For $r_1>r_2>\cdots>r_t\geq -1$ of the same parity and $t$ is even, we know the induced representation 
		\[|det(\cdot)|^{\frac{r_1-r_2}{4}}\tau_{\frac{r_1+r_2}{2}}\times |det(\cdot)|^{\frac{r_3-r_4}{4}}\tau_{\frac{r_3+r_4}{2}}\times\cdots\times|det(\cdot)|^{\frac{r_{t-1}-r_t}{4}}\tau_{\frac{r_{t-1}+r_t}{2}} \]
		is irreducible (see \cite{bernstein1977induced,zelevinsky1980induced}).
	\end{enumerate}
	To be precise, in light of ($\star$), viewing $\sigma_{\bar{r}}\hookrightarrow |det(\cdot)|^{\frac{r_1-r_2}{4}}\tau_{\frac{r_1+r_2}{2}}\rtimes \sigma_{\bar{r'}}$ with
	\[\sigma_{\bar{r'}}\leftrightarrow\phi_{\sigma_{\bar{r'}}}=\phi_\tau\otimes S_{r_3}\oplus\phi_\tau\otimes S_{r_4}\oplus\cdots\oplus\phi_\tau\otimes S_{r_t}\oplus\phi_\sigma. \]
	We have the following reduced decomposition of $M_{c,a}(s,~\tau,~\sigma_{\bar{r}})$ via the embedding
	
	\underline{Way 3: $\sigma_{\bar{r}}\hookrightarrow|det(\cdot)|^{\frac{r_1-r_2}{4}}\tau_{\frac{r_1+r_2}{2}}\rtimes \sigma_{\bar{r'}}$}:
	\[\xymatrix{|det(\cdot)|^{s}\rho_{c}(\tau_a)\rtimes \sigma_{\bar{r}} \ar@{^(->}[r] \ar[ddd]_{M_{c,a}(s,\tau,\sigma_{\bar{r}})} & \mbox{$\underbrace{|det(\cdot)|^{s}\rho_{c}(\tau_a)\times|det(\cdot)|^{\frac{r_1-r_2}{4}}\tau_{\frac{r_1+r_2}{2}}}$}\rtimes \sigma_{\bar{r'}}\ar[d]^{M_{GL}(\cdots)}\\
		&|det(\cdot)|^{\frac{r_1-r_2}{4}}\tau_{\frac{r_1+r_2}{2}}\times\mbox{$\underbrace{|det(\cdot)|^{s}\rho_{c}(\tau_a)\rtimes\sigma_{\bar{r'}}}$} \ar[d]^{M_{c,a}(s,\tau,\sigma_{\bar{r'}})}\\ 
		&\mbox{$\underbrace{|det(\cdot)|^{\frac{r_1-r_2}{4}}\tau_{\frac{r_1+r_2}{2}}\times|det(\cdot)|^{-s}\rho_{c}(\tau_a)}$}\rtimes \sigma_{\bar{r'}}  \ar[d]^{M_{GL}(\cdots)}\\
		|det(\cdot)|^{-s}\rho_{c}(\tau_a)\rtimes \sigma_{\bar{r}} \ar@{^(->}[r] & |det(\cdot)|^{-s}\rho_{c}(\tau_a)\times|det(\cdot)|^{\frac{r_1-r_2}{4}}\tau_{\frac{r_1+r_2}{2}}\rtimes\sigma_{\bar{r'}}.
	}\]
	Via Lemma \ref{lemGL} + ($\star\star$) + (A)(B) + (a)(b), one can do an easy calculation of normalization factors of those intertwining operators and see that they also match each other, thus the reduction step holds.
	
	\underline{\bf Step 3 }
	The main idea is to apply an induction argument on the pair $(c,a)$ under the partial order given by
	\[(c',a')<(c,a)\mbox{ if at least one of }c'\leq c \mbox{ and }a'\leq a\mbox{ is a strict less than.} \]
	We will return to the base cases, i.e., Case $c=1$ and Case $a=1$, in the next two sections. Note that $r_1>a>r_2$ and $c,~a>1$, we calculate the discrepancies of two ways of reduced decompositions of $M_{c,a}(s,~\tau,~\sigma_{r})$ corresponding to the following two embeddings Way $1$ and Way $1'$.
	
	\underline{Way 1: $\rho_c(\tau_a)\hookrightarrow|det(\cdot)|^{-\frac{1}{2}}\rho_{c-1}(\tau_{a})\times |det(\cdot)|^{\frac{c-1}{2}}\tau_a$}:
	\[\xymatrix{ |det(\cdot)|^{s}\rho_c(\tau_a)\rtimes \sigma_r \ar@{^(->}[r] \ar[ddd]_{M_{c,a}(s,\tau,\sigma_r)} & |det(\cdot)|^{s-\frac{1}{2}}\rho_{c-1}(\tau_{a})\times\mbox{$\underbrace{|det(\cdot)|^{s+\frac{c-1}{2}}\tau_a\rtimes \sigma_r}$}\ar[d]^{M_{1,a}(s+\frac{c-1}{2},\tau,\sigma_r)}\\
		&\mbox{$\underbrace{|det(\cdot)|^{s-\frac{1}{2}}\rho_{c-1}(\tau_{a})\times|det(\cdot)|^{-s-\frac{c-1}{2}}\tau_a}$}\rtimes \sigma_r  \ar[d]^{M_{GL}(\cdots)}\\
		&|det(\cdot)|^{-s-\frac{c-1}{2}}\tau_a\times\mbox{$\underbrace{|det(\cdot)|^{s-\frac{1}{2}}\rho_{c-1}(\tau_{a})\rtimes\sigma_r}$} \ar[d]^{M_{c-1,a}(s-\frac{1}{2},\tau,\sigma_r)}\\ 
		|det(\cdot)|^{-s}\rho_c(\tau_a)\rtimes \sigma_r \ar@{^(->}[r] & |det(\cdot)|^{-s-\frac{c-1}{2}}\tau_a\times|det(\cdot)|^{-s+\frac{1}{2}}\rho_{c-1}(\tau_{a})\rtimes\sigma_r.
	}\] 
	Given the above commutative diagram, we have $M_{c,a}(s,~\tau,~\sigma_{r})=$ $$M_{c-1,a}\left(s-\frac{1}{2},~\tau,~\sigma_{r}\right)\circ M_{GL}\left((s-\frac{1}{2},-s-\frac{c-1}{2}),~\rho_{c-1}(\tau_a)\otimes \tau_a \right)\circ M_{1,a}\left(s+\frac{c-1}{2},~\tau,~\sigma_{r}\right).$$ Here $$M_{GL}\left((s-\frac{1}{2},-s-\frac{c-1}{2}),~\rho_{c-1}(\tau_a)\otimes \tau_a \right)$$ is the standard intertwining operator associated to $$\rho_{c-1}(\tau_a)|det(\cdot)|^{s-\frac{1}{2}}\times \tau_a|det(\cdot)|^{-s-\frac{c-1}{2}}.$$
	By the induction assumption and our Proposition \ref{propGL}, we have 
	\[M^*_{c-1,a}\left(s-\frac{1}{2},~\tau,~\sigma_{r}\right)\circ M^*_{GL}\left((s-\frac{1}{2},-s-\frac{c-1}{2}),~\rho_{c-1}(\tau_a)\otimes \tau_a \right)\circ M^*_{1,a}\left(s+\frac{c-1}{2},~\tau,~\sigma_{r}\right)\]
	is holomorphic for all $s\in \mathbb{C}$. Here $M^*_{GL}(\cdots):=\alpha_{GL}(\cdots)^{-1}M_{GL}(\cdots)$ with \[\alpha_{GL}(\cdots)=L\left(2s,\tau_a\times \tau_a\right).\]
	Thus $M^*_{c,a}(s,~\tau,~\sigma_{r})=\frac{\alpha_{c-1,a}(s-\frac{1}{2},~\tau,~\sigma_{r})\alpha_{GL}(\cdots)\alpha_{1,a}(s+\frac{c-1}{2},~\tau,~\sigma_{r})}{\alpha_{c,a}(s,~\tau,~\sigma_{r})}\times$
	\[M^*_{c-1,a}\left(s-\frac{1}{2},~\tau,~\sigma_{r}\right)\circ M^*_{GL}\left((s-\frac{1}{2},-s-\frac{c-1}{2}),~\rho_{c-1}(\tau_a)\otimes \tau_a \right)\circ M^*_{1,a}\left(s+\frac{c-1}{2},~\tau,~\sigma_{r}\right).\]
	So the possible poles of $M^*_{c,a}(s,~\tau,~\sigma_{r})$ are controlled by \[P_1:=\frac{\alpha_{c-1,a}(s-\frac{1}{2},~\tau,~\sigma_{r})\alpha_{GL}(\cdots)\alpha_{1,a}(s+\frac{c-1}{2},~\tau,~\sigma_{r})}{\alpha_{c,a}(s,~\tau,~\sigma_{r})}\]
	which equals, notice that $L(s, \tau_a\times \tau_a)=L(s,\tau_a,\rho)L(s,\tau_a,\rho^-)$,
	\[\prod_{j=1}^{c-2}L(2s+j,\tau_a\times \tau_a)\begin{cases}
	L(2s,\tau_a,\rho^-)L(2s+c-1,\tau_a,\rho)L(s+\frac{c-1}{2},\tau_a\times \sigma_{r}), &\mbox{ if $c$ is odd;}\\
	L(2s,\tau_a,\rho)L(2s+c-1,\tau_a,\rho)L(s+\frac{c-1}{2},\tau_a\times \sigma_{r}), &\mbox{ if $c$ is even.}
	\end{cases}\]
	In view of the fact that
	$$L(2s,\tau_a,\rho)=\prod_{i=1}^{\lceil\frac{a}{2}\rceil}L(2s+1+a-2i,\tau,\rho)\prod_{i=1}^{\lfloor\frac{a}{2}\rfloor}L(2s+c-2i,\tau,\rho^-),$$ we then know that the real parts of poles of $P_1$ are non-positive numbers.
	
	\underline{Way 1': $\rho_c(\tau_a)\hookrightarrow|det(\cdot)|^{\frac{1}{2}}\rho_c(\tau_{a-1})\times |det(\cdot)|^{-\frac{a-1}{2}}\rho_c(\tau)$}:
	\[\xymatrix{ |det(\cdot)|^{s}\rho_c(\tau_a)\rtimes \sigma_r \ar@{^(->}[r] \ar[ddd]_{M_{c,a}(s,\tau,\sigma_r)} & |det(\cdot)|^{s+\frac{1}{2}}\rho_{c}(\tau_{a-1})\times\mbox{$\underbrace{|det(\cdot)|^{s-\frac{a-1}{2}}\rho_c(\tau)\rtimes \sigma_r}$}\ar[d]^{M_{c,1}(s-\frac{a-1}{2},\tau,\sigma_r)}\\
		&\mbox{$\underbrace{|det(\cdot)|^{s+\frac{1}{2}}\rho_{c}(\tau_{a-1})\times|det(\cdot)|^{-s+\frac{a-1}{2}}\rho_c(\tau)}$}\rtimes \sigma_r  \ar[d]^{M_{GL}(\cdots)}\\
		&|det(\cdot)|^{-s+\frac{a-1}{2}}\rho_c(\tau)\times\mbox{$\underbrace{|det(\cdot)|^{s+\frac{1}{2}}\rho_{c}(\tau_{a-1})\rtimes\sigma_r}$} \ar[d]^{M_{c,a-1}(s+\frac{1}{2},\tau,\sigma_r)}\\ 
		|det(\cdot)|^{-s}\rho_c(\tau_a)\rtimes \sigma_r \ar@{^(->}[r] & |det(\cdot)|^{-s+\frac{a-1}{2}}\rho_c(\tau)\times|det(\cdot)|^{-s-\frac{1}{2}}\rho_{c}(\tau_{a-1})\rtimes\sigma_r.
	}\] 
	Regarding the above commutative diagram, we decompose $M_{c,a}(s,~\tau,~\sigma_{r})$ into the following product
	\[M_{c,a-1}(s,~\tau,~\sigma_{r})\circ M_{GL}\left((s+\frac{1}{2},-s+\frac{a-1}{2}),~\rho_{c}(\tau_{a-1})\otimes \rho_{c}(\tau)\right)\circ M_{c,1}(s-\frac{a-1}{2},~\tau,~\sigma_{r}). \]
	By our Proposition \ref{propGL}, we have $M^*_{GL}(\cdots):=\alpha_{GL}(\cdots)^{-1}M_{GL}(\cdots)$ is holomorphic for $s\in \mathbb{C}$ with
	\[\alpha_{GL}(\cdots):=\prod_{j=0}^{c-1}L(2s-j,\tau\times \tau). \]
	On the other hand, by the induction assumption, we know that 
	\[M^*_{c,a-1}(s,~\tau,~\sigma_{r})\circ M^*_{GL}\left((s+\frac{1}{2},-s+\frac{a-1}{2}),~\rho_{c}(\tau_{a-1})\otimes \rho_{c}(\tau)\right)\circ M^*_{c,1}(s-\frac{a-1}{2},~\tau,~\sigma_{r}) \]
	is holomorphic for $s\in \mathbb{C}$, and $M^*_{c,a}(s,~\tau,~\sigma_{r})=\frac{\alpha_{c,a-1}(s+\frac{1}{2},~\tau,~\sigma_{r})\alpha_{GL}(\cdots)\alpha_{c,1}(s-\frac{a-1}{2},~\tau,~\sigma_{r})}{\alpha_{c,a}(s,~\tau,~\sigma_{r})}\times $
	\[M^*_{c,a-1}(s,~\tau,~\sigma_{r})\circ M^*_{GL}\left((s+\frac{1}{2},-s+\frac{a-1}{2}),~\rho_{c}(\tau_{a-1})\otimes \rho_{c}(\tau)\right)\circ M^*_{c,1}(s-\frac{a-1}{2},~\tau,~\sigma_{r}). \]
	Thus the possible poles of $M^*_{c,a}(s,~\tau,~\sigma_{r})$ is governed by \[P_{1'}:=\frac{\alpha_{c,a-1}(s+\frac{1}{2},~\tau,~\sigma_{r})\alpha_{GL}(\cdots)\alpha_{c,1}(s-\frac{a-1}{2},~\tau,~\sigma_{r})}{\alpha_{c,a}(s,~\tau,~\sigma_{r})}\]
	which equals, $[\cdots]$ means may appear, 
	\begin{align*}
	[L(s-\frac{a-r_2}{2}-\frac{c-1}{2},\tau\times \tau)]&\prod_{i=1}^{\lceil\frac{c}{2}\rceil}L(2s-c-1+2i,\tau,\rho^-)L\left(2(s-\frac{a-1}{2})-c-1+2i,\tau,\rho\right) \\
	\times [L(s-\frac{a-1}{2}-\frac{c-1}{2},\tau\times \sigma)]&\prod_{i=1}^{\lfloor\frac{c}{2}\rfloor}L(2s-c+2i,\tau,\rho)L\left(2(s-\frac{a-1}{2})-c+2i,\tau,\rho^-\right). 
	\end{align*}
	It is easy to see that the real parts of poles of $P_{1'}$ are non-negative numbers.
	
	Comparing $P_1$ and $P_{1'}$, we see that the common poles are at $Re(s)=0$ and given by
	\[\begin{cases}
	L(2s,\tau,\rho^-), &\mbox{ if $c$ is odd;}\\
	L(2s,\tau,\rho), &\mbox{ if $c$ is even.}
	\end{cases}\]
	To get rid of the above possible poles, we show directly that $M^*_{c,a}(0,~\tau,~\sigma_{r})$ is holomorphic. This follows from the following facts.
	\begin{enumerate}[(i).]
		\item $\rho_{c}(\tau_a)\rtimes \sigma_{r}$ is multiplicity-free. This is known by M{\oe}glin (cf. \cite{moeglin2006paquets,moeglin2011multiplicityone}).
		\item $M^*_{c,a}(0,~\tau,~\sigma_{r})^2=id.$, up to a non-zero scalar. This follows from a simple calculation of the analogue normalization factor $\beta_{c,a}(s,~\tau,~\sigma_{r})$.
	\end{enumerate}
	To be precise, by (i), write $\rho_{c}(\tau_a)\rtimes \sigma_{r}=\oplus_i\pi_i$ with $\pi_i\not\simeq\pi_j$ for any $i\neq j$, then by (ii), we must have $M^*_{c,a}(0,~\tau,~\sigma_{r})|_{\pi_i}=id.$, up to a non-zero scalar. Whence our reduction step holds provided that the base cases are proved in the following sections.
\end{proof}	

\section{proof of main theorem \ref{mainthm}: case $\rho_{c}(\tau)\rtimes \sigma_{r}$}\label{case1}

\begin{proof}[Proof of Main Theorem \ref{mainthm} (Case $\rho_{c}(\tau)\rtimes \sigma_{r}$)]As usual, our proof involves an induction argument. In what follows, we will first discuss the initial cases, i.e., 
	\[\mbox{Case }|det(\cdot)|^s\tau\rtimes \sigma_{r}\mbox{ and Case }|det(\cdot)|^s\rho_{c}(\tau)\rtimes \sigma. \] 
	Notice that the former case, proved in \cite{luo2021casselmanshahidiconjecture}, is also an initial case in Section \ref{case2} and will be discussed there, so here we only prove the latter case as follows.
	
	\underline{\bf Step 1 }(Initial step for induction). Recall that we have the following reduced decompositions of 
	\[w_c:=w_{c;1}=(-1)^{kc}\begin{pmatrix}
	&  & I_{kc} \\ 
	& I_{N_0} &  \\ 
	\pm I_{kc} &  & 
	\end{pmatrix} . \]
	with respect to partitions of $c$ 
	\[w_c=w_{c-j}w_{c-j,j}w_{j}, \]
	where $j=1$ or $c-1$ and
	\[w_{c-j,j}:=w_{c-j,j;1}=\begin{pmatrix}
	& I_{kj} &  &  &  \\ 
	I_{k(c-j)} &  &  &  &  \\ 
	&  & I_{N_0} &  &  \\ 
	&  &  &  & I_{k(c-j)} \\ 
	&  &  & I_{kj} & 
	\end{pmatrix}. \]
	and their corresponding discrepancies of normalization factors
	\begin{align*}
	P_1:&=\frac{\alpha_{c-1}(s-\frac{1}{2},~\tau,~\sigma)\cdot \alpha_{GL}(2s-\frac{1}{2}+\frac{c-1}{2},~\tau)\cdot \alpha(s+\frac{c-1}{2},~\tau,~\sigma)}{ \alpha_{c}(s,~\tau,~\sigma)}\\
	&=\begin{cases}
	L(2s,\tau,\rho^-)L(2s+c-1,\tau,\rho)L(s+\frac{c-1}{2},\tau\times\sigma),  & \mbox{$c$ is odd};\\
	L(2s,\tau,\rho)L(2s+c-1,\tau,\rho)L(s+\frac{c-1}{2},\tau\times\sigma),  & \mbox{$c$ is even}.
	\end{cases}\\
	P_{2}:&=\frac{\alpha(s-\frac{c-1}{2},~\tau,~\sigma)\cdot \alpha_{GL}(2s+\frac{1}{2}-\frac{c-1}{2},~\tau)\cdot \alpha_{c-1}(s+\frac{1}{2},~\tau,~\sigma)}{ \alpha_{c}(s,~\tau,~\sigma)}.\\
		&=\begin{cases}
	L(2s-c+2,\tau,\rho)L(2s+1,\tau,\rho)L(s-\frac{c-3}{2},\tau\times\sigma),  & \mbox{$c$ is even};\\
	L(2s-c+2,\tau,\rho)L(2s+1,\tau,\rho^-)L(s-\frac{c-3}{2},\tau\times\sigma),  & \mbox{$c$ is odd}.
	\end{cases}
	\end{align*}
	It is readily to see that if $(P_1^{-1},~P_{2}^{-1})=1$, i.e., co-prime in $\mathbb{C}[q^{-s}]$, then the holomorphy of $M^*_c(s,~\tau,~\sigma)$ follows easily from induction on $c$, while the case $c=1$ is just a rephrase of the definition of $L$-factors by the Langlands--Shahidi theory. Then the remaining issue is to rule out those common poles. One can compare terms $P_1$ with $P_{2}$ to get the following possible common poles.
	\[c=2\mbox{ and }s=0\mbox{ or }s=-\frac{1}{2},\qquad c=3\mbox{ and }s=0, \]
	where the last case occurs if and only if
	\[L(s,\tau\times\sigma) \mbox{ and }L(2s,\tau,\rho^-) \mbox{ both have a pole at }s=0. \]
	Now let us start to rule out those points case-by-case as done in Lemma \ref{lemGL}.
	
	\underline{$c=2$ and $s=-\frac{1}{2}$}:
	\[\xymatrix@C+3pc{|det(\cdot)|^{-1}\tau\times \mbox{$\underbrace{\tau\rtimes \sigma}$}\ar[r]^-{\mbox{scalar}}_-{\mbox{simple pole}} & \mbox{$\underbrace{|det(\cdot)|^{-1}\tau\times \tau}$}\rtimes\sigma\ar[d]^{\mbox{subrep}\mapsto 0}_{\mbox{simple zero}}\\
		\tau\times |det(\cdot)|\tau\rtimes \sigma & \tau\times \mbox{$\underbrace{ |det(\cdot)|^{-1}\tau\rtimes \sigma}$}\ar[l] }\] 
	Here we can see that $|det(\cdot)|^{-\frac{1}{2}}\rho_2(\tau)\rtimes\sigma\hookrightarrow |det(\cdot)|^{-1}\tau\times \tau\rtimes \sigma\longmapsto 0$, thus we get the holomorphy of $M_2(s,~\tau,~\sigma)$ at this point.
	
	\underline{$c=2$ and $s=0$}:
	\[\xymatrix@C+3pc{|det(\cdot)|^{-\frac{1}{2}}\tau\times \mbox{$\underbrace{|det(\cdot)|^{\frac{1}{2}}\tau\rtimes \sigma}$}\ar[r]^{\mbox{subrep}\mapsto 0}_{\mbox{simple zero}} & \mbox{$\underbrace{|det(\cdot)|^{-\frac{1}{2}}\tau\times |det(\cdot)|^{-\frac{1}{2}}\tau}$}\rtimes\sigma\ar[d]^-{\mbox{scalar}}_-{\mbox{simple pole}}\\
		|det(\cdot)|^{-\frac{1}{2}}\tau\times |det(\cdot)|^{\frac{1}{2}}\tau\rtimes \sigma & |det(\cdot)|^{-\frac{1}{2}}\tau\times \mbox{$\underbrace{ |det(\cdot)|^{-\frac{1}{2}}\tau\rtimes \sigma}$}\ar[l]^{\mbox{subrep}\mapsto 0}_{\mbox{simple zero}} }\]
	Arguing as above, we see that $M_2(s,~\tau,~\sigma)$ is holomorphic at $s=0$. 
	
	\underline{$c=3$ and $s=0$}: 
	\[\xymatrix@C+3pc{|det(\cdot)|^{-1}\tau\times \tau\times \mbox{$\underbrace{ |det(\cdot)|^{1}\tau\rtimes \sigma}$} \ar[r]^{\mbox{subrep}\mapsto 0}_{\mbox{simple zero}} & |det(\cdot)|^{-1}\tau\times \mbox{$\underbrace{ \tau\times |det(\cdot)|^{-1}\tau}$}\rtimes \sigma \ar[d]^{\mbox{subrep}\mapsto 0}_{\mbox{simple zero}} \\
		|det(\cdot)|^{-1}\tau\times |det(\cdot)|^{-1}\tau\times \mbox{$\underbrace{\tau\rtimes \sigma}$} \ar[d]^{\mbox{scalar}}_{\mbox{simple pole}} & \mbox{$\underbrace{ |det(\cdot)|^{-1}\tau\times |det(\cdot)|^{-1}\tau}$}\times \tau\rtimes \sigma \ar[l]^-{\mbox{scalar}}_-{\mbox{simple pole}} \\
		|det(\cdot)|^{-1}\tau\times  \mbox{$\underbrace{|det(\cdot)|^{-1}\tau\times \tau}$} \rtimes \sigma\ar[r]^{\mbox{subrep}\mapsto 0}_{\mbox{simple zero}} & |det(\cdot)|^{-1}\tau\times \tau\times \mbox{$\underbrace{ |det(\cdot)|^{-1}\tau\rtimes \sigma}$} \ar[d]^{\mbox{subrep}\mapsto 0}_{\mbox{simple zero}} \\
		& |det(\cdot)|^{-1}\tau\times \tau\times |det(\cdot)|^{1}\tau\rtimes \sigma }\]
	One can see that the product of the middle four intertwining operators is a scalar, but with a simple pole, then consider the product of it with the remaining two intertwining operators, we obtain a non-zero scalar map. Thus $M_3(s,~\tau,~\sigma)$ is holomorphic at $s=0$. Whence we have completed the proof of the holomorphy of $M^*_c(s,~\tau,~\sigma)$ for all $s$.
	\begin{rem}
		One may notice that the holomoporphy of $M^*_c(s,~\tau,~\sigma)$ at $Re(s)=0$ follows similarly from an argument in the last part of Section \ref{redsteps} using a general multiplicity-free property of certain unitary induced representations, but we would like to keep it here for its elementariness and in case we need it in a future work.
	\end{rem}
	
	\underline{\bf Step 2 }(Induction step). Recall that $\sigma_{r}$ corresponds to the following Langlands parameter \[\phi_{\sigma_{r}}=\phi_{\tau}\otimes S_{r_1}\oplus\phi_{\tau}\otimes S_{r_2}\oplus\phi_{\sigma},\]
	where $r_1>r_2\geq -1$ of the same parity. Notice that we can decompose $M_c(s,~\tau,~\sigma_r)$ in terms of the following commutative diagram.	
	
	\underline{Way 3}: Viewing $\sigma_r\hookrightarrow|det(\cdot)|^{\frac{r_1-r_2}{4}}\tau_{\frac{r_1+r_2}{2}}\rtimes \sigma$, it gives rise to
	\[\xymatrix{|det(\cdot)|^{s}\rho_{c}(\tau)\rtimes \sigma_{r} \ar@{^(->}[r] \ar[ddd]_{M_c(s,\tau,\sigma_{r})} & \mbox{$\underbrace{|det(\cdot)|^{s}\rho_{c}(\tau)\times|det(\cdot)|^{\frac{r_1-r_2}{4}}\tau_{\frac{r_1+r_2}{2}}}$}\rtimes \sigma\ar[d]^{M_{GL}(\cdots)}\\
		&|det(\cdot)|^{\frac{r_1-r_2}{4}}\tau_{\frac{r_1+r_2}{2}}\times\mbox{$\underbrace{|det(\cdot)|^{s}\rho_{c}(\tau)\rtimes\sigma}$} \ar[d]^{M_c(s,\tau,\sigma)}\\ 
		&\mbox{$\underbrace{|det(\cdot)|^{\frac{r_1-r_2}{4}}\tau_{\frac{r_1+r_2}{2}}\times|det(\cdot)|^{-s}\rho_{c}(\tau)}$}\rtimes \sigma  \ar[d]^{M_{GL}(\cdots)}\\
		|det(\cdot)|^{-s}\rho_{c}(\tau)\rtimes \sigma_{r} \ar@{^(->}[r] & |det(\cdot)|^{-s}\rho_{c}(\tau)\times|det(\cdot)|^{\frac{r_1-r_2}{4}}\tau_{\frac{r_1+r_2}{2}}\rtimes\sigma.
	}\]
	Via Lemma \ref{lemGL} + (A)(B), one can calculate the normalization factors of intertwining operators and obtain the discrepancy $P_3$ associated to Way 3 as follows.
	\begin{align*}
	P_3:&=\frac{\alpha_{GL}(\cdots)\alpha_c(s,~\tau,~\sigma)\alpha_{GL}(\cdots)}{\alpha_c(s,~\tau,~\sigma_r)}\\
	&=\begin{cases}
	1, &\mbox{$r_2>0$},\\
	L\left(s-\frac{c-r_2}{2},\tau\times \tau\right),& \mbox{$r_2= 0$},\\
	L\left(s-\frac{c-1}{2},\tau\times \sigma\right)L\left(s-\frac{c-r_2}{2},\tau\times \tau\right),& \mbox{$r_2< 0$}.
	\end{cases}
	\end{align*}
	It is easy to see that the real parts of possible poles of $P_3$ are non-negative numbers. On the other hand, regarding the reduced decomposition given by $w_c=w_{c-1}w_{c-1,1}w_1$, i.e., Way 1 in the above {\bf Step 1 }(Initial step for induction), we know that the corresponding discrepancy $P_1$ of normalization factors is given by
	\begin{align*}
    P_1:&=\frac{\alpha_{c-1}(s-\frac{1}{2},~\tau,~\sigma_r)\cdot \alpha_{GL}(2s-\frac{1}{2}+\frac{c-1}{2},~\tau)\cdot \alpha(s+\frac{c-1}{2},~\tau,~\sigma_r)}{ \alpha_{c}(s,~\tau,~\sigma_r)}\\
    &=\begin{cases}
    L(2s,\tau,\rho^-)L(2s+c-1,\tau,\rho)L(s+\frac{c-1}{2},\tau\times\sigma_r),  & \mbox{$c$ is odd};\\
    L(2s,\tau,\rho)L(2s+c-1,\tau,\rho)L(s+\frac{c-1}{2},\tau\times\sigma_r),  & \mbox{$c$ is even}.
    \end{cases}
    \end{align*}
    It is also easy to see that the real parts of possible poles of $P_1$ are non-positive numbers. Comparing $P_1$ and $P_3$, we obtain that it suffices to prove the holomorphy of $M^*_c(s~\tau,~\sigma_{r})$ at $Re(s)=0$. Notice that this can be proved in the same way as done in the last part of Section \ref{redsteps}, therefore the induction step holds. Whence we complete the proof of this base case.
\end{proof}

\section{proof of main theorem \ref{mainthm}: case $\tau_a\rtimes \sigma_{r}$}\label{case2}
As usual, the argument relies on the intrinsic non-symmetry property of normalization factors corresponding to different reduced decompositions of Weyl elements as seen in Proposition \ref{propGL}. One may notice that a proof of this case has been given in \cite{luo2021casselmanshahidiconjecture}, but a simplified argument is present in the following for completeness.
\begin{proof}[Proof of Main Theorem \ref{mainthm} (Case $\tau_a\rtimes\sigma_{r}$)] Recall that $r_1>a>r_2$ and the generic discrete series $\sigma_{r}$ corresponds to the Langlands parameter \[\phi_{\sigma_{r}}=\phi_{\tau}\otimes S_{r_1}\oplus\phi_{\tau}\otimes S_{r_2}\oplus \phi_{\sigma}.\]
	
	\underline{\bf Step 1 }(Induction step). Notice that $a>1$, we have the following reduced decomposition of $M(s,~\tau_a,~\sigma_{r})$ corresponding to the embedding 
	
	\underline{Way 1': $\tau_a\hookrightarrow|det(\cdot)|^{\frac{1}{2}}\tau_{a-1}\times |det(\cdot)|^{-\frac{a-1}{2}}\tau$}:
	\[\xymatrix{ |det(\cdot)|^{s}\tau_a\rtimes \sigma_r \ar@{^(->}[r] \ar[ddd]_{M(s,\tau_a,\sigma_r)} & |det(\cdot)|^{s+\frac{1}{2}}\tau_{a-1}\times\mbox{$\underbrace{|det(\cdot)|^{s-\frac{a-1}{2}}\tau\rtimes \sigma_r}$}\ar[d]^{M(s-\frac{a-1}{2},\tau,\sigma_r)}\\
		&\mbox{$\underbrace{|det(\cdot)|^{s+\frac{1}{2}}\tau_{a-1}\times|det(\cdot)|^{-s+\frac{a-1}{2}}\tau}$}\rtimes \sigma_r  \ar[d]^{M_{GL}(\cdots)}\\
		&|det(\cdot)|^{-s+\frac{a-1}{2}}\tau\times\mbox{$\underbrace{|det(\cdot)|^{s+\frac{1}{2}}\tau_{a-1}\rtimes\sigma_r}$} \ar[d]^{M(s+\frac{1}{2},\tau_{a-1},\sigma_r)}\\ 
		|det(\cdot)|^{-s}\tau_a\rtimes \sigma_r \ar@{^(->}[r] & |det(\cdot)|^{-s+\frac{a-1}{2}}\tau\times|det(\cdot)|^{-s-\frac{1}{2}}\tau_{a-1}\rtimes\sigma_r.
	}\] 
	Via Lemma \ref{lemGL} + (A)(B) + ($\star\star$), one can carry out the calculation for the corresponding normalization factors and the discrepancy $P_{1'}$ associated to Way $1'$ as follows.
	\begin{align*}
	P_{1'}:&=\begin{cases}
	L(2s,\tau,\rho^-)L(2s-(a-1),\tau,\rho)L\left(s-\frac{a-r_2}{2},\tau\times \tau\right)L\left(s-\frac{a-1}{2},\tau\times \sigma\right), & \mbox{if $a$ odd and $r_2>0$};\\
	L(2s,\tau,\rho)L(2s-(a-1),\tau,\rho)L\left(s-\frac{a-r_2}{2},\tau\times \tau\right)L\left(s-\frac{a-1}{2},\tau\times \sigma\right), & \mbox{if $a$ even and $r_2>0$};\\
	L(2s,\tau,\rho^-)L(2s-(a-1),\tau,\rho), & \mbox{if $a$ odd and $r_2\leq 0$};\\
	L(2s,\tau,\rho)L(2s-(a-1),\tau,\rho), & \mbox{if $a$ even and $r_2\leq 0$}.
	\end{cases}
	\end{align*}
	Thus it is easy to see that the real parts of possible poles of $P_{1'}$ are non-negative numbers. On the other hand, it is well-known that the intertwining operator $M(s,~\tau_a,~\sigma_r)$ is always well-defined for $Re(s)>0$, thus if we could prove its holomorphy at $Re(s)=0$, then the holomorphy of $M^*(s,~\tau_a,~\sigma_r)$ is established by induction on $a$. Now we prove the holomorphy of $M^*(s,~\tau_a,~\sigma)$ at $s=0$ directly. This follows easily from a similar argument utilized in the last part of Section \ref{redsteps} by the fact that $M^*(0,~\tau_a,~\sigma)^2=id.$, up to a non-zero scalar, and the fact that $\tau_a\rtimes \sigma$ is multiplicity-free (see \cite[P. 272 Remark]{luo2020knapp} in general). Whence we finish the proof of Case $\tau_a\rtimes\sigma_{r}$ provided that the initial case $|det(\cdot)|^s\tau\rtimes \sigma_{r}$ is proved.
	
	\underline{\bf Step 2 }(Initial step for induction). Indeed, there are two initial cases, i.e.,
	\[\mbox{Case }|det(\cdot)|^s\tau\rtimes \sigma_{r}\mbox{ and Case }|det(\cdot)|^s\tau_a\rtimes \sigma. \]
	But the latter case has been discussed in the above {\bf Step 1} (Induction step), so we only need to consider the former case as follows.
	
	\underline{The case $|det(\cdot)|^{s}\tau\rtimes \sigma_r$ with $r_1>1>r_2$}: In this case, we have $r_2=0$ if $r_1$ is even, $-1$ otherwise. Then we know that 
	\[\sigma_r\hookrightarrow |det(\cdot)|^{\frac{r_1}{4}}\tau_{\frac{r_1}{2}}\rtimes\sigma\mbox{ (if $r_1$ even)}; \quad \sigma_r\hookrightarrow |det(\cdot)|^{\frac{r_1+1}{4}}\tau_{\frac{r_1-1}{2}}\rtimes\sigma\mbox{ (if $r_1$ odd)},\]
	which implies that \[\alpha(s,~\tau,~\sigma_r)=L(2s,\tau,\rho)L(s,\tau\times \tau_{r_1})=L(2s,\tau,\rho)L\left(s+\frac{r_1-1}{2},\tau\times \tau\right).\]
	Note that we can decompose $M(s,\tau,\sigma_r)$ in terms of the following commutative diagram.	
	
	\underline{Way 3}: Viewing $\sigma_r\hookrightarrow|det(\cdot)|^{\frac{r_1-r_2}{4}}\tau_{\frac{r_1+r_2}{2}}\rtimes \sigma$, it gives rise to
	\[\xymatrix{|det(\cdot)|^{s}\tau\rtimes \sigma_{r} \ar@{^(->}[r] \ar[ddd]_{M(s,\tau,\sigma_{r})} & \mbox{$\underbrace{|det(\cdot)|^{s}\tau\times|det(\cdot)|^{\frac{r_1-r_2}{4}}\tau_{\frac{r_1+r_2}{2}}}$}\rtimes \sigma\ar[d]^{M_{GL}(\cdots)}\\
		&|det(\cdot)|^{\frac{r_1-r_2}{4}}\tau_{\frac{r_1+r_2}{2}}\times\mbox{$\underbrace{|det(\cdot)|^{s}\tau\rtimes\sigma}$} \ar[d]^{M(s,\tau,\sigma)}\\ 
		&\mbox{$\underbrace{|det(\cdot)|^{\frac{r_1-r_2}{4}}\tau_{\frac{r_1+r_2}{2}}\times|det(\cdot)|^{-s}\tau}$}\rtimes \sigma  \ar[d]^{M_{GL}(\cdots)}\\
		|det(\cdot)|^{-s}\tau\rtimes \sigma_{r} \ar@{^(->}[r] & |det(\cdot)|^{-s}\tau\times|det(\cdot)|^{\frac{r_1-r_2}{4}}\tau_{\frac{r_1+r_2}{2}}\rtimes\sigma.
	}\]
	Via Lemma \ref{lemGL} + (A)(B), one can calculate the normalization factors of intertwining operators and obtain the discrepancy $P_3$ associated to Way 3 as follows.
	\begin{align*}
	P_3:&=\frac{\alpha_{GL}(\cdots)\alpha(s,~\tau,~\sigma)\alpha_{GL}(\cdots)}{\alpha(s,~\tau,~\sigma_r)}\\
	&=L(s,\tau\times \sigma)L\left(s+\frac{r_2-1}{2},\tau\times \tau\right).
	\end{align*}
	As $r_2=0$ or $-1$, we can see that the real parts of possible poles of $P_3$ are non-negative numbers. A similar argument as in the above {\bf Step 1} (Induction step) shows that $M^*(s,~\tau,~\sigma_{r})$ is holomorphic for $s\in\mathbb{C}$.	
\end{proof}

\paragraph*{\textbf{Acknowledgments}} The author would like to thank Eyal Kaplan for his kindness and help. Thanks are also due to the referee for his/her detailed comments. The author was supported by the ISRAEL SCIENCE FOUNDATION Grant Number 376/21.
			
\bibliographystyle{amsalpha}
\bibliography{ref}

			
\end{document}